\documentclass[10pt]{amsart} 
\usepackage{amsmath} 
\usepackage{amssymb} 
\usepackage{amsthm} 
\usepackage{mathtools}
\usepackage{comment}
\usepackage{tikz}

\usetikzlibrary{cd}
\usepackage[colorlinks]{hyperref}
\usepackage{enumitem}
\usepackage[normalem]{ulem}

\numberwithin{equation}{section}

\newcommand{\sub}{\subseteq}

\newcommand{\R}{\mathbb{R}}

\newcommand{\N}{\mathbb{N}}
\newcommand{\eps}{\varepsilon}

\newcommand{\M}{\mathcal M}

\newcommand{\orb}{\mathrm{orb}}
\newcommand{\dist}{\mathrm{dist}}

\numberwithin{chap}{section}
\newtheorem{thm}{Theorem}
\numberwithin{thm}{section}
\newtheorem{conj}[thm]{Conjecture}
\newtheorem{prop}[thm]{Proposition}
\newtheorem{defn}[thm]{Definition}
\newtheorem{lem}[thm]{Lemma}

\newtheorem{cor}[thm]{Corollary}

\newtheorem{condition}[thm]{Condition}

\DeclarePairedDelimiter{\norm}{\lVert}{\rVert}

\makeatletter
\let\oldnorm\norm
\def\norm{\@ifstar{\oldnorm}{\oldnorm*}}

\makeatother

\begin{document}

\pagestyle{myheadings} \thispagestyle{empty} \markright{}
\title{Decoupling for surfaces with radial symmetry}

\author{Jianhui Li and Tongou Yang}
\address[Jianhui Li]{Department of Mathematics, Northwestern University\\
Evanston, IL 60208, United States
}
\email{jianhui.li@northwestern.edu}

\address[Tongou Yang]{Department of Mathematics, University of California\\
Los Angeles, CA 90095, United States}
\email{tongouyang@math.ucla.edu\\ 
tomyangcuhk@gmail.com}

\date{}

\begin{abstract}
    We utilise the two principles of decoupling introduced in \cite{LiYang2024} to prove decoupling for two types of surfaces exhibiting radial symmetry. The first type are surfaces of revolution in $\R^n$ generated by smooth surfaces in $\R^3$. The second type of surfaces are graphs of trivariate homogeneous smooth functions of nonzero degree.
\end{abstract}

\maketitle

\section{Introduction}

Introduced by Wolff \cite{Wolff2000} and further developed by numerous mathematicians including the work of \cite{LP06,LW02, Bo13Tri}, the revolutionary Fourier analytic tool, decoupling, has been applied to tackle a wide range of problems in subjects such as harmonic analysis, number theory, additive combinatorics and partial differential equations.

In the breakthrough work of \cite{BD2015,BD2017}, Bourgain-Demeter proved sharp decoupling inequalities for $C^2$ surfaces in $\R^n$ with nonzero Gaussian curvature. In \cite{BDG2016}, Bourgain-Demeter-Guth (see also \cite{shortproofmomentcurve}) proved sharp decoupling for $C^2$ curves in $\R^n$ with nonzero torsion.

A natural follow-up is to understand to what extent these decoupling results can be generalised to manifolds with vanishing curvature. Current literature in this direction include, but not limited to, \cite{BGLSX,Demeter2020,BDK2019,Yang2,LiYang,LiYang2023,Kemp1,Kemp2,Kemp2024,GMO24,LiYang2024}.

In \cite{LiYang2024}, we formulated two principles of decoupling, which are mainly based on the Pramanik-Seeger iterations \cite{PS2007}. Continuing from \cite{LiYang2024}, in this paper, we prove sharp decoupling inequalities for two classes of surfaces, namely, surfaces of revolution in $\R^n$ generated by smooth surfaces in $\R^3$, and graphs of tri-variate homogeneous smooth functions of nonzero degree. See our main Theorems \ref{thm:main_revolution} and \ref{thm:homo} below.

\subsection{Decoupling inequalities for surfaces in $\R^n$}

Let us first formulate the decoupling inequalities in this article.

\begin{defn}\label{defn:decoupling}
    Given a compact subset $S\sub \R^n$ and a finite collection $\mathcal R$ of parallelograms $R\sub \R^n$ with a bounded overlapping function\footnote{See Definition \ref{defn:enlarged_overlap}.}. For Lebesgue exponents $p,q\in [2,\infty]$, we define the $\ell^q(L^p)$-decoupling constant $\mathrm{Dec}(S,\mathcal R,p,q)$ to be the smallest constant $\mathrm{Dec}$ such that
    \begin{equation}\label{eq:defn_decoupling}
        \norm{\sum_{R}f_R}_{L^p(\R^n)} \leq \mathrm{Dec}\,\, (\#\mathcal R)^{\frac{1}{2}-\frac{1}{q}}\norm{\norm{f_R}_{L^p(\R^n)}}_{\ell^q(R\in \mathcal R)}
    \end{equation}
    for all smooth test function $f_R$ Fourier supported on $R\cap S$.
    
    Given a subset $S\sub \R^n$, we say that $S$ can be $\ell^q(L^p)$ decoupled into the parallelograms $R\in \mathcal R$ at the cost of $K$, if $S\sub \cup \mathcal R$ and $\mathrm{Dec}(S,\mathcal R,p,q)\le K$.
\end{defn}
We also recommend that the reader refer to \cite{LiYang2024} for more general formulations of decoupling.

Since we often deal with decoupling for graphs of functions, we introduce the following notation for simplicity.
\begin{itemize}
    \item For a subset $S\sub \R^n$ and $\delta>0$, we denote by $N_\delta(S)$ the $\delta$-neighbourhood of $S$, namely, 
    \begin{equation*}
        N_\delta(S):=\{x+c:x\in S,|c|\le \delta\}.
    \end{equation*}
    \item For a subset $\Omega\sub \R^{n-1}$, $\phi:\Omega\to \R$ be $C^2$ and $\delta>0$, we denote by $N^\phi_\delta(\Omega)$ the $\delta$-{\it vertical neighbourhood} of $\Omega$, namely, 
\begin{equation*}
    N^\phi_\delta(\Omega):=\{(x,y):x\in \Omega,|y-\phi(x)|\le \delta\}.
\end{equation*}
\end{itemize}
\begin{defn}\label{defn:graphical_decoupling}
    Let $\Omega\sub \R^k$, $\phi:\Omega\to \R$ be $C^2$. Let $\delta>0$. 
    \begin{itemize}
        \item We say that $\omega\sub \Omega$ is $(\phi,\delta)$-flat, or $\phi$ is $\delta$-flat over $\omega$, if there exists an affine function $\Lambda$ such that
\begin{equation*}
    \sup_{x\in \omega}|\phi(x)-\Lambda(x)|\le \delta.
\end{equation*}
In this case, $N_\delta^\phi(\omega)$ is equivalent to a parallelogram (see Subsection \ref{sec:equivalence_objects}).

\item We say that a parallelogram $\Omega_0\sub \Omega$ can be $\phi$-$\ell^q(L^p)$ decoupled into $(\phi,\delta)$-flat parallelograms $\omega$ at scale $\delta$ (at cost $C$), if $N_\delta^\phi(\Omega_0)$ can be $\ell^q(L^p)$ decoupled into $N_\delta^\phi(\omega)$ which are almost parallelograms\footnote{By decoupling into almost parallelograms $R$, we mean decoupling into actual parallelograms equivalent to $R$.} (at cost $C$). When the exponents $p,q$ are clear from the context, we simply say that $\Omega_0$ can be $\phi$-decoupled into parallelograms $\omega$ at scale $\delta$.
    \end{itemize}
    
For a parallelogram $R\sub \R^n$, we refer to the smallest dimension of $R$ by its {\it width} (see \eqref{eqn:width} for the precise definition).    
    
\end{defn}

\subsection{Main decoupling results}\label{sec_main_decoupling_result}
We are now ready to state our main theorems in this paper.
\begin{thm}[Decoupling for surfaces of revolution generated by $2$-surfaces]\label{thm:main_revolution_original}
    Let $M_0$ be a compact piece of a smooth surface in $\R^3$. For $n> 3$, let $H$ be a two-dimensional subspace of $\R^n$, with positive distance from $M_0$ (naturally embedded into $\R^n$). We rotate $M_0$ about $H$ to obtain a surface $M\sub \R^n$ of codimension $1$. Then for every $0<\delta\ll_{M}1$, $\eps\in (0,1]$ and $2\le p\le \frac{2(n+1)}{n-1}$, $N_\delta(M)$ can be $\ell^p(L^p)$ decoupled into a family $\mathcal P_{M,\eps,\delta}$ of 
    parallelograms in $\R^n$ of width at least $\delta$ in the sense of Definition \ref{defn:decoupling}, at cost $C_\eps \delta^{-\eps}$. Here $C_\eps$ and the overlap function of $\mathcal P_{M,\eps,\delta}$ depend only on $M,\eps$.
\end{thm}

As a special case, we obtain a corresponding decoupling result for radial functions, generalising \cite{BDK2019} and \cite[Corollary 1.4]{LiYang2024}.
\begin{cor}[Decoupling for radial functions]\label{cor_radial}
    Let $A_0$ denote the annulus $\{x\in \R^{n-1}:1\le |x|\le 2\}$, and let $\phi:A_0\to \R$ be smooth. Assume $\phi$ is radial, namely, $\phi(x)=\phi(|x|)$ for every $1\le |x|\le 2$. Then for every $0<\delta\ll_\phi 1$ and $2\le p\le \frac{2(n+1)}{n-1}$, $A_0$ can be $\phi$-$\ell^p(L^p)$ decoupled into a family $\mathcal P_{\phi,\delta}$ of almost 
    parallelograms in $\R^{n-1}$ at scale $\delta$ and cost $C_\eps \delta^{-\eps}$ in the sense of Definition \ref{defn:graphical_decoupling}, for every $\eps\in (0,1)$. Here, each parallelogram in $\mathcal P_{\phi,\delta}$ has width $\gtrsim_\phi \delta^{1/2}$, the
    overlap function of $\mathcal P_{\phi,\delta}$ depends only on $\phi$, and the cost $C_\eps$ depends only on $\phi,\eps$. 
    
    Moreover, if $D^2\phi$ is positive semidefinite, then this $\ell^p(L^p)$ can be strengthened to $\ell^2(L^p)$ decoupling. The same holds true if $\phi$ is radial and smooth in the unit ball.
\end{cor}

By a trivial partition argument, Theorem \ref{thm:main_revolution_original} can be rephrased into the following version.
\begin{thm}\label{thm:main_revolution}
    Consider a surface $M$ parametrised by either one of the following forms:
     \begin{equation}\label{eqn:TypeI_intro}
        \text{(Type I) }\quad \left\{\left(s,r(s)t,r(s)\sqrt{1-|t|^2}\right):s\in [-1,1]^2,t\in \R^l,|t|\le 1/2
        \right\},
    \end{equation}
    where $r:[-1,1]^2\to [1,2]$ is a smooth function, or
    \begin{equation}\label{eqn:TypeII_intro}
        \text{(Type II) }\quad \left\{\left(s_1,x(s_1,r),rt,r\sqrt{1-|t|^2}\right):s_1\in [-1,1],r\in [1,2],t\in \R^l,|t|\le 1/2\right\},
    \end{equation}
    where $x:[-1,1]^2\to \R$ is a smooth function. Then for every $0<\delta\ll_{M}1$, $\eps\in (0,1]$ and $2\le p\le \frac{2(4+l)}{2+l}$, $N_\delta(M)$ can be $\ell^p(L^p)$ decoupled into a family $\mathcal P_{M,\eps,\delta}$ of parallelograms of width at least $\delta$ in the sense of Definition \ref{defn:decoupling} at cost $C_{\eps} \delta^{-\eps}$, where $C_\eps$ and the overlap function of $\mathcal P_{M,\eps,\delta}$ depend only on $M,\eps$. 
    
    Moreover, if $M$ has a positive semidefinite second fundamental form and is of Type I, then this $\ell^p(L^p)$ decoupling can be upgraded to $\ell^2(L^p)$ decoupling.
\end{thm}

It may happen that a given surface has both parts of Type I and parts of Type II. A typical case is when $M=[-1,1]\times S\sub \R^4$, where $S$ is part of the torus in $\R^3$ parametrised by
\begin{equation*}
    S:=\{(\sin \theta,(2+\cos \theta)\cos \varphi, (2+\cos \theta)\sin \varphi),\,\, \theta\in [0,\pi/2],\,\,\varphi\in [\pi/6,\pi/3]\}.
\end{equation*}
It can be checked that the part near the circle $M\cap \{\theta=0\}$ is of Type I, where $l=1$ and $r(s)=r(s_2)=2+\cos (\arcsin s_2)$, and the part near the circle $M\cap \{ \theta=\pi/2\}$ is of Type II, where $l=1$ and $x(s_1,r)=x(r)=\sin (\arccos (r-2))$.

We will actually prove a slight generalisation of Theorem \ref{thm:main_revolution}, that is, Theorem \ref{thm:main_general_revolution}. 

As we will see, the proof for decoupling of Type I surfaces will be much easier than that of Type II surfaces, and we are able to prove a $\ell^2(L^p)$ decoupling inequality in the case where a Type I surface is convex. We also highly anticipate that a $\ell^2(L^p)$ decoupling inequality holds for convex Type II surfaces, but our proof technique, more precisely, Theorem \ref{thm:general_sublevel_set}, fails to achieve this.

Apart from surfaces of revolution, we are also able to prove a decoupling result for tri-variate homogeneous functions, which generalises \cite{LiYang}.
\begin{thm}[Decoupling for tri-variate homogeneous functions]\label{thm:homo}
Let $A_0$ denote the annulus $\{x\in \R^{3}:1\le |x|\le 2\}$, and let $\phi:A_0\to \R$ be smooth. Assume that $\phi$ is (positively) homogeneous of nonzero degree, namely, there exists some $\alpha\in \R\backslash\{0\}$ such that $\phi(rx)=r^\alpha \phi(x)$ for every $r\in [1,2]$ and every $|x|=1$. Then for every $0<\delta\ll_{M}1$, $\eps\in (0,1]$ and $2\le p\le \frac{10}{3}$, $A_0$ can be $\phi$-$\ell^p(L^p)$ decoupled into a family $\mathcal P_{\phi,\eps,\delta}$ of parallelograms of width at least $\delta$ at scale $\delta$ in the sense of Definition \ref{defn:graphical_decoupling} and cost $C_{\eps}\delta^{-\eps}$, where $C_\eps$ and the overlap function of $\mathcal P_{\phi,\eps,\delta}$ depend only on $\phi,\eps$. In particular, we have
\begin{itemize}
    \item If $\phi$ is (positively) homogeneous (of a positive degree) and smooth in the unit cube $[-1,1]^3$, then the decoupling holds for $[-1,1]^3$. In particular, if $\phi$ is a homogeneous polynomial of degree at most $d$, then $C_{\eps}$ can be chosen to depend only on $\eps,d$.
    \item If $D^2\phi$ is positive semidefinite, then this $\ell^p(L^p)$ decoupling can be upgraded to $\ell^2(L^p)$ decoupling.
\end{itemize}
\end{thm}
Examples of tri-variate homogeneous functions of nonzero degree include the function $\phi(x)=|x|^\gamma$ where $\gamma\ne 0$, all nonconstant homogeneous polynomials like $\phi(x,y,z)=x^2y+xz^2$, and homogeneous rational functions like $\phi(x,y,z)=\frac {xyz}{x^4+y^4+z^4}$. However, we remark that our result does not prove decoupling for homogeneous functions of degree $0$, such as $\phi(x,y,z)=\frac{xyz^2}{x^4+y^4+z^4}$.

\subsection{Decoupling for smooth surfaces of vanishing Gaussian curvature}

The main theorems above are partial progress towards the following conjecture.

\begin{conj}\label{conj:dec}
    Let $S$ be a compact piece of a smooth surface in $\R^n$. Then for every $2\le p\le \frac{2(n+1)}{n-1}$ and $0<\delta\ll_S 1$, $N_\delta(M)$ can be $\ell^p(L^p)$ decoupled into parallelograms (with width at least $\delta$). Moreover, if $S$ is convex, then this $\ell^p(L^p)$ decoupling can be upgraded to $\ell^2(L^p)$ decoupling.
\end{conj}
The case $n=2$ was solved by 
\cite{Yang2} (see also \cite{BGLSX,Demeter2020}), and the case $n=3$ was solved by \cite{LiYang2023} and \cite{GMO24}. The case $n\ge 4$ remains open in general. To the authors' knowledge, only decoupling in the case of the cone $\{(t,|t|):1\le |t|\le 2\}$ (\cite{BD2015}) and some special surfaces (see, for instance, \cite{GLZZ2022}) have been recorded in the literature. The difficulty of the decoupling problem mainly lies in the complicated structure of the zero sets of the Gaussian curvature, even when the surface is given by a polynomial.

\subsection{Main ideas}
We now describe the main ideas of the proof of Theorems \ref{thm:main_revolution} and \ref{thm:homo}.
\subsubsection{Induction on scales}

Decoupling inequalities of the form \eqref{eq:defn_decoupling} work well with induction on scales arguments. Let $\mathrm{Dec}_M(1,\delta)$ be the cost of decoupling a surface $M$ from scale $1$ to scale $\delta$. A common strategy is to first decouple the surface into $\sigma$-flat rectangles at some intermediate scale $\sigma\in (\delta,1)$. Each $\sigma$-flat rectangle is further decoupled into $\delta$-flat rectangles. Thus, we have the following bound:
\begin{equation}\label{eqn:induction_intro}
    \mathrm{Dec}_M(1,\delta) \leq \mathrm{Dec}_M(1,\sigma) \mathrm{Dec}_M(\sigma,\delta).
\end{equation}

Suppose that $\mathrm{Dec}_M(\sigma,\delta) \lesssim_\varepsilon \left(\frac{\sigma}{\delta}\right)^{-\varepsilon}$. A simple bootstrapping argument shows the desired estimate $\mathrm{Dec}_M(1,\delta) \lesssim \delta^{-O(\varepsilon)}$.

\subsubsection{Radial principle of decoupling}
Fix $\varepsilon>0$, $\delta\in (0,1]$ and $\sigma=\delta^{1-\eps}$. Following the idea of Pramanik-Seeger \cite{PS2007}, the radial principle of decoupling developed in \cite{LiYang2024} exploits the fact that if a part of a surface parametrised by
$$
\Pi = \{(s,r(s)t,r(s)\sqrt{1-|t|^2}: s\in S,t\in T\} \quad r(s)\sim 1
$$
is within the $\sigma$-neighbourhood of a hyperplane, then it 
can be approximated by $\delta$-neighbourhood of 
$$
\Sigma = \{(s,\tau,br(s)+\sqrt{1-|\tau|^2}):s\in S , \tau = r(s)t, t\in T\}, \quad b=b(S,T)\sim 1.
$$
Since $N_\delta(\Sigma)$ is roughly $N_\delta(\Pi)$, $\mathrm{Dec}_\Sigma(\sigma,\delta)\sim\mathrm{Dec}_\Pi(\sigma,\delta)$. By the discussion above, it suffices to show that $\mathrm{Dec}_\Sigma(\sigma,\delta) \lesssim_\varepsilon \left(\frac{\sigma}{\delta}\right)^{-\varepsilon}$. Thus, the radial principle of decoupling principle reduces the study of $\Pi$ to that of $\Sigma$.

\subsubsection{Rescaling invariance}

Decoupling inequalities formulated as in Conjecture \ref{conj:dec} favours rescaling. It is straightforward to see from Definition \ref{defn:decoupling} that rescaling on the frequency side factors out equally on both sides. If $\Omega$ is a $(\phi,\delta)$-flat parallelogram as in Definition \ref{defn:graphical_decoupling}, then $\Xi(\Omega)$ is $(\phi\circ \Xi^{-1}, \delta)$-flat for any affine bijection $\Xi$. For neighbourhoods of parametric surfaces, the same is true if rescaling happens only along tangential directions. 

The key observation is that the $\delta$-neighbourhood of a $\sigma$-flat piece can be rescaled to the $\delta/\sigma$-neighbourhood of some other surface $\tilde{M}$. We say that a family $\mathcal{M}$ is rescaling invariant if $\tilde{M}$ generated from a $\sigma$-flat piece in $M \in \mathcal{M}$ still belongs to the family. Let $\mathrm{Dec}_\mathcal{M}$ denote the supremum of the corresponding decoupling constants $\mathrm{Dec}_M$ over all $M\in \mathcal M$. Then, \eqref{eqn:induction_intro} becomes
$$
\mathrm{Dec}_\mathcal{M}(1,\delta) \leq \mathrm{Dec}_\mathcal{M}(1,\sigma) \mathrm{Dec}_\mathcal{M}(\sigma,\delta) \leq \mathrm{Dec}_\mathcal{M}(1,\sigma)  \mathrm{Dec}_\mathcal{M}(1,\delta/\sigma).
$$
If we naively put the induction hypothesis $\mathrm{Dec}_\mathcal{M}(1,\alpha) \leq C_\varepsilon \alpha^{-\varepsilon}$ for all $\alpha>2\delta$, we arrive at
$$
\mathrm{Dec}_\mathcal{M}(1,\delta) \leq C_\varepsilon^{2} \delta^{-\varepsilon},
$$
which is almost enough. It turns out that we only need to do slightly better in one of the terms. 

\subsubsection{Surfaces of revolution} 
We describe the main idea of proof of Theorem \ref{thm:main_revolution}. Let $K:M \mapsto \R$ measure the Gaussian curvature of $M$. By Bourgain-Demeter's decoupling inequality \cite{BD2015,BD2017}, the decoupling on regions on $M$ where $K$ is away from $0$ is known. For surfaces of revolution $M\sub \R^n$ generated by a 2-surface $M_0\sub \R^3$, the Gaussian curvature factors into $K_rK_a$, where $K_r$ corresponds to the principal curvatures along the radial direction, which is the Gaussian curvature of $M_0$, and $K_a$ measures the principal curvatures generated along the angular direction. The case $K_a \neq 0$ corresponds to parts of Type I \eqref{eqn:TypeI_intro}, which can be handled directly with the radial principle of decoupling. Thus, it suffices to reduce the problem of decoupling parts of Type II \eqref{eqn:TypeII_intro} to that of Type I. 

The reduction uses the degeneracy locating principle introduced in \cite{LiYang2024}. In this particular case, we shall embed our surfaces of revolution into a larger family of surfaces that respects rescaling, while keeping the same curvature factorisation phenomenon. We first show that the sublevel set on which $|K_a|$ is smaller than some $\sigma$ can be decoupled into rectangular boxes that are ready to rescale. The sublevel set problem is typically easier because it is one-dimensional lower than the original problem. After rescaling, $|K_a|$ is small on the entire surface, forcing the rescaled surface to be in the $\sigma$-neighbourhood of some surfaces of Type II but one-dimensional lower. Since both sublevel set decoupling and lower dimensional decoupling are known, we have
$$
\mathrm{Dec}_{M_{\deg}}(1,\sigma) \lesssim_\varepsilon \sigma^{-\varepsilon/2} \implies \mathrm{Dec}_{M_{\deg}}(1,\sigma) \leq \sigma^{-\varepsilon}\quad \forall \varepsilon>0,
$$
where $\mathrm{Dec}_{M_{\deg}}$ roughly denotes the cost of decoupling on the region where $|K_a| \leq \sigma\ll 1$. By doing slightly better, it is enough to close the induction.

\subsubsection{Homogeneous function}
To prove Theorem \ref{thm:homo}, the key observation is that if a homogeneous function $\phi(x) \sim 1$ on the annulus $|x| \sim 1$, it can be reparametrised into some surface of the form $\Pi$. Thus, it suffices to reduce the case of a general homogeneous function $\phi$ to this special case, which then can be handled by a sublevel set decoupling followed by a rescaling argument.

\subsection{Outline of the paper} 
In Section \ref{sec_preliminaries}, we give the preliminaries on surfaces of revolution in higher dimensions. In Section \ref{sec_tools}, we invoke previously established decoupling theorems in \cite{BD2015,BD2017,LiYang2024} as our basic tools. In Section \ref{sec:typeI}, we prove Theorem \ref{thm:main_revolution_original} in the Type I case. In Sections \ref{sec:pseudo-polynomials} to \ref{sec:nondegenerate_case}, we prove Theorem \ref{thm:main_revolution_original} in the Type II case. In Section \ref{sec_homo_nonzero}, we prove Theorem \ref{thm:homo}. Section \ref{sec:appendix} is the appendix, which discusses further generalisations (in addition to \cite{LiYang2024}) of the uniform decoupling theorem for bivariate polynomials proved in \cite{LiYang2023}.

\subsection{Notation}\label{sec:notation}

\begin{enumerate}
    \item \label{item:big_O} We use the standard notation $a=O_M(b)$, or $|a|\lesssim_M b$ to mean that there is a constant $C$ depending on some parameter $M$, such that $|a|\leq Cb$. When the dependence on the parameter $M$ is unimportant, we may simply write $a=O(b)$ or $|a|\lesssim b$. Similarly, we define $\gtrsim$ and $\sim$.

\item Given a subset $S\sub \R^n$, we define the $\delta$-neighbourhood $N_\delta(S)$ to be
    \begin{equation*}
N_\delta(S) := \{ x + c: x \in M, |c|<\delta\}.
\end{equation*}

    \item Given a continuous function $\phi:\Omega\sub \R^{n-1}\to \R$, we define the $\delta$-vertical neighbourhood $N^\phi_\delta(\Omega)$ by    
    \begin{equation*}
N_\delta^\phi(\Omega): =\{ (x , \phi(x) + c): x \in \Omega, c \in [-\delta,\delta]\}.
\end{equation*}

\item Given a collection $\mathcal R$ of subsets of $\R^n$, we abbreviate $\cup \mathcal R$ to be the subset of $\R^n$ defined by $\cup_{R\in \mathcal R}R$.

\item We say that an affine transformation $L:\R^n\to \R^m$ given by $Lx=Ax+b$ is bounded by $C>0$, if every entry of $A,b$ is bounded by $C$. We say that an affine bijection $\lambda$ on $\R^n$ is {\it bounded} by $C$ if $\lambda x=Ax+b$ where all entries of $A,b$ are bounded by $C$.

\item \label{item:Pnd} We denote by $\mathcal P_{n,d}$ the collection of all $n$-variate real polynomials $\phi$ of degree at most $d$, such that $\sup_{[-1,1]^n}|\phi(x)|\le 1$.

\item We say that a $C^2$ function is convex/concave if its Hessian matrix is positive semidefinite/negative semidefinite, respectively. We say that a $C^2$ surface is convex if all of its nonzero principal curvatures are of the same sign. Thus, a surface given by the graph of a $C^2$ function is convex if and only if the function is either convex or concave.

\end{enumerate}

\subsubsection{Notation about parallelograms}\label{sec:equivalence_objects}
Decoupling inequalities are formulated using parallelograms in $\R^n$. To deal with technicalities, it is crucial that we make clear some fundamental geometric concepts.

\begin{enumerate}
    \item A parallelogram $R\sub \R^n$ is defined by a set of the form
    \begin{equation*}
        \{x+b\in \R^n:|x\cdot u_i|\le l_i\},
    \end{equation*}
    where $\{u_i:1\le i\le n\}$ is a basis of $\R^n$ consisting of unit vectors, $b\in \R^n$ and $l_i\ge 0$. If $\{u_i:1\le i\le n\}$ is also orthogonal, we say that $R$ is a rectangle. Unless otherwise specified, we always assume that a parallelogram in $\R^n$ has positive $n$-volume, namely, $l_i>0$ for all $1\le i\le n$. We say that two parallelograms are disjoint if their interiors are disjoint. 
    
    \noindent We define the {\it width} of a parallelogram $P$ to be 
    \begin{equation}\label{eqn:width}
        \inf\{\text{smallest dimension of $T$: $T$ is a rectangle containing $P$} \}.
    \end{equation}
    By the John ellipsoid theorem \cite{John}, every convex body in $\R^n$ is $C_n$-equivalent to a rectangle in $\R^n$, for some dimensional constant $C_n$. For this reason and in view of the enlarged overlap introduced in Definition \ref{defn:enlarged_overlap}, we often do not distinguish rectangles from parallelograms.

    \item Let $R\sub \R^n$ be a parallelogram. If there is no confusion, we often denote by $c(R)$ the centre of $R$. We denote by $\lambda_R$ an affine bijection from $[-1,1]^n$ to $R$. (There is more than one way to make such a choice of $\lambda_R$, but in our paper, any such choice will work identically, since in decoupling inequalities, the orientations of such affine bijections do not make any difference.)

    \item Unless otherwise specified, for a parallelogram $S\sub \R^n$ and $C>0$, we denote by $CS$ the concentric dilation of $S$ by a factor of $C$. Given $t\in \R^n$, we denote $S+t=\{s+t:s\in S\}$. Note that in this notation, we have $C(S+t)=CS+t$.
    
\item For $0<\delta<1$, by a {\it tiling} of a parallelogram $R\sub \R^n$ by cubes of side length $\delta$, we mean a covering $\mathcal T$ of $R$ by translated copies $T$ of a cube of side length $\delta$, such that different $T$ and $T'$ from $\mathcal T$ have disjoint interiors, and that $T\sub 2R$.  

    \item Given a subset $S\sub \R^n$, a parallelogram $R\sub \R^n$ and $C\ge 1$, we say that $S,R$ are $C$-equivalent if $C^{-1}R\sub S\sub CR$. If the constant $C$ is unimportant, we simply say that $S,R$ are equivalent, denoted $S\approx R$.

\end{enumerate}

\subsubsection{Enlarged overlap}
Let $\mathcal R$ be a family of parallelograms. In all decoupling inequalities in the literature, we hope that the parallelograms in $\mathcal R$ do not overlap too much. Since we often deal with constant enlargement of parallelograms, we will need a slightly stronger version of the overlap condition.

\begin{defn}\label{defn:enlarged_overlap}
By an {\it overlap function} we mean an increasing function $B:[1,\infty)\to [1,\infty)$. Given a finite family $\mathcal R$ of parallelograms in $\R^n$ and an overlap function $B$, we say that $\mathcal R$ is $B$-overlapping, or $B$ is an overlap function of $\mathcal R$, if for every constant $\mu\ge 1$ we have
    \begin{equation*}
        \sum_{R\in \mathcal R}1_{\mu R}\le B(\mu).
    \end{equation*}
\end{defn}
In the context of decoupling, the function $B$ may depend on many parameters such as $\eps$ and the family of the manifolds we are decoupling. We say $\mathcal R$ is {\it boundedly overlapping} if its overlap function $B$ is independent of $\delta$. For technical treatments on overlap functions, we refer the reader to \cite[Section 5.6]{LiYang2024}.

\subsection{Acknowledgement}
The first author thanks Betsy Stovall and Xiumin Du for preliminary discussions on an early stage of the work. The second author was supported by the Croucher Fellowships for Postdoctoral Research. The authors thank Terence Tao for helpful discussions.

\section{Preliminaries of surfaces of revolution}\label{sec_preliminaries}

In this section, we provide the preliminaries of surfaces of revolutions in higher dimensions and study the Gaussian curvatures of such surfaces. We will actually work with slightly more general surfaces than exact surfaces of revolution, as in Theorem \ref{thm:main_general_revolution} below. 

\subsection{Rotation in higher dimensions}
In $\R^n$, fix a point $p=(p_1,\dots,p_n)$ and a $k$-dimensional subspace $H$. We define the orbit of $p$ revolving about $H$ to be
\begin{equation*}
    \mathrm{orb}_H(p)=\{q\in \R^n:\dist(p,H)=\dist(q,H),(p-q)\perp H\}.
\end{equation*}
We study some trivial cases first. If $p=0$, then $\orb_H(p)=\{0\}$, so in the following we assume $p\ne 0$. If $H=\R^n$, then $\orb_H(p)=\{p\}$, so in the following we assume $H$ is a proper subspace of $\R^n$. If $H=\{0\}$, then $\orb_H(p)$ is the $(n-1)$-sphere centred at $0$ of radius $|p|$. If $H$ is a $(n-1)$-plane, then $\orb_H(p)=\{p,-p\}$.

Now we come to some nontrivial cases. Let $n\ge 3$, $1\le k\le n-2$, and assume without loss of generality that $H$ is the $x_1\dots x_k$ plane. Then for every nonzero $p\in \R^n$, we have
\begin{equation*}
    \orb_H(p)=\left\{q\in \R^n:q_i=p_i \,\,\forall 1\le i\le k,\quad \sum_{i=k+1}^n q_i^2=\sum_{i=k+1}^n p_i^2\right\}.
\end{equation*}
Thus, $\orb_H(p)$ is a manifold of dimension $n-k-1$. By symmetry, we can partition $\orb_H(p)$ into $O(1)$ smaller pieces, and just consider the subset
\begin{equation*}
    \left\{\left(p_1,\dots,p_k,rt,r\sqrt{1-|t|^2}\right):|t|\le 1/2\right\},
\end{equation*}
where $t=(t_1,\dots,t_{n-k-1})$ and $r=\sqrt{\sum_{i=k+1}^n p_i^2}$.

\subsection{Rotation of a manifold}
Let $n\ge 3$, $1\le k\le n-1$, and assume without loss of generality that $H$ is the $x_1\dots x_k$ plane. Suppose that we are given a manifold in $\R^n$ of dimension $m\in [1,n-1]$, parametrised by
\begin{equation*}
    (x_1(s),\dots,x_n(s)):s\in [-1,1]^m,
\end{equation*}
where each $x_i(s)$ is $C^2$ and satisfies that $\left(\frac{\partial x_i(s)}{\partial s_j}\right)_{1\le i\le n,1\le j\le m}$ has rank $m$. The image of the manifold under the revolution about $H$ is the union of all orbits $\orb_H(p)$, where $p$ ranges through the manifold. More precisely, after cutting into $O(1)$ pieces, it suffices to study the image of revolution given by the piece
\begin{equation*}
    \left\{\left(x_1(s),\dots,x_k(s), r(s)t,r(s)\sqrt{1-|t|^2}\right):s\in [-1,1]^m,|t|\le \frac 1 2\right\},
\end{equation*}
where $t=(t_1,\dots,t_{n-k-1})$ and $r(s)=\sqrt{\sum_{i=k+1}^n x_i(s)^2}$. As such, it suffices to ignore the coordinates $x_i(s)$, $k+1\le i\le n$, and just study the radius function $r(s)$.

We assume that $r(s)$ is away from $0$, so that it is always a $C^2$
function. Assume also that the rank of the matrix $(\frac{\partial (x_i(s),r(s))}{\partial s_j})_{1\le i\le k,1\le j \le m}$ is a constant $m\le k$. (Thus, we are generating a surface of revolution from a manifold in a lower dimensional subspace.)

We will actually study slightly more general manifolds, parametrised by
\begin{equation}\label{eqn_surface}
    \left\{\left(x_1(s),\dots,x_k(s), r(s)t,r(s)\psi(t)\right):s\in [-1,1]^m, t\in [-1,1]^{l}\right\},
\end{equation}
where we have denoted
\begin{equation*}
    l=n-k-1,
\end{equation*}
and $\psi(t)$ is a {\it revolution function} defined right below.
\begin{defn}\label{defn:general_psi}
    Let $l\in \N$ and define a linear operator $L$ on $C^2([-1,1]^l)$ by
\begin{equation}\label{eqn_Lpsi}
    L\psi(t):=\psi(t)-t\cdot \nabla \psi(t).
\end{equation}
We say that a $C^2$ function $\psi:[-1,1]^l\to \R$ is a revolution function, if 
\begin{equation}
\begin{aligned}
    &\psi(0)=1, \quad \nabla \psi(0)=0,\\
    &\inf_{t\in [-1,1]^l}L\psi(t)>0,\quad
    \inf_{t\in [-1,1]^l} |\det D^2 \psi(t)|>0.   
\end{aligned}
\end{equation}
\end{defn}
For instance, the functions
\begin{equation}\label{eqn:examples_of_psi}
    \psi_1(t)=\sqrt{1-\frac{|t|^2}{2l}}, \quad \psi_2(t):=1-|t|^2,\quad \psi_3(t):=1-t_1^2+t_2^2/2-4t_3^2
\end{equation}
are revolution functions. The function $\psi_1$ corresponds to surfaces of revolution after rescaling\footnote{The rescaling is not important, and the sole purpose of it is to make sure that $\psi_1$ is defined on $[-1,1]^l$, which simplifies subsequent notation.} $t\mapsto \sqrt {2l}t$. Note that $\psi_1$ and $\psi_2$ are strictly concave functions. 

The first task is to determine whether the above parametrisation \eqref{eqn_surface} is actually (locally) a manifold in $\R^n$.
\begin{prop}\label{prop_surface_condition}
Assume we are in either of the following cases:
\begin{enumerate}
    \item The rank of the matrix $(\frac{\partial x_i(s)}{\partial s_j})_{1\le i\le k,1\le j \le m}$ is $m\le k$, or
    \item $L\psi(t)\ne 0$ for any $t\in [-1,1]^l$.
\end{enumerate}
Then the image of the parametrisation in \eqref{eqn_surface} is locally a $C^2$ manifold in $\R^n$ of dimension $m+l=m+n-k-1$.
\end{prop}
\begin{proof}
    By direct computation, the Jacobian of the parametrisation in \eqref{eqn_surface} is given by
    \begin{equation}\label{eqn:Jacobian}
        \begin{bmatrix}
        \frac {\partial x_1}{\partial s_1} & \cdots & \frac {\partial x_k}{\partial s_1} & \frac {\partial r}{\partial s_1}t_1 & \cdots &\frac {\partial r}{\partial s_1}t_{l} & \frac {\partial r}{\partial s_1}\psi\\
        \vdots & \ddots & \vdots & \vdots & \ddots & \vdots & \vdots\\
        \frac {\partial x_1}{\partial s_m} & \cdots & \frac {\partial x_k}{\partial s_m} & \frac {\partial r}{\partial s_m}t_1 & \cdots &\frac {\partial r}{\partial s_m}t_{l} & \frac {\partial r}{\partial s_m}\psi\\[3pt]
        0 & \cdots & 0 & r & \cdots & 0 & r\frac{\partial\psi}{\partial t_1}\\
        \vdots & \ddots & \vdots & \vdots & \ddots & \vdots & \vdots\\
        0 & \cdots & 0 & 0 & \cdots & r & r\frac{\partial\psi}{\partial t_{l}}
    \end{bmatrix}.
    \end{equation}
    Since $r$ is away from $0$, the last $l$ rows are linearly independent. By the rank assumption on the matrix $\left(\frac{\partial (x_i(s),r(s))}{\partial s_j}\right)_{1\le i\le k,1\le j \le m}$ and the assumption that $\psi(0)=1$, the first $m$ rows are linearly independent. 
    \begin{enumerate}
        \item If we are in the first case, then the rows of the whole Jacobian are linearly independent, in view of the zero submatrix in the lower left corner.
        \item If we are in the second case, then by elementary column operations, we may turn the first $m$ entries of the last column of \eqref{eqn:Jacobian} into $(\nabla r) L\psi(t)$. By the rank assumption on \eqref{eqn:Jacobian}, taking $t=0$ and by a trivial partition, we may assume that the matrix
        \begin{equation*}
            \begin{bmatrix}
                \frac {\partial x_1}{\partial s_1} & \cdots & \frac {\partial x_{m-1}}{\partial s_1} & \frac {\partial r}{\partial s_1}\\
                \vdots & \ddots & \vdots & \vdots\\
                \frac {\partial x_1}{\partial s_m} & \cdots & \frac {\partial x_{m-1}}{\partial s_m} & \frac {\partial r}{\partial s_m}
            \end{bmatrix}
        \end{equation*}
        is invertible. By relabelling if necessary, after suitable elementary row operations, we may assume that the above matrix is just $I_m$. By exactly the aforesaid elementary row operations applied to the first $m$ rows of the matrix in \eqref{eqn:Jacobian}, it can be reduced to the form
        \begin{equation*}
            \begin{bmatrix}
                I_{m-1} & ? & O & 0\\
                \vec 0 & ? & t & \psi\\
                O & O & rI_{l} & r\nabla \psi
            \end{bmatrix}.
        \end{equation*}
        Further elementary row operations lead to
        \begin{equation*}
            \begin{bmatrix}
                I_{m-1} & ? & O & 0\\
                \vec 0 & ? & \vec 0 & L\psi\\
                O & O & I_{l} & \nabla \psi
            \end{bmatrix}.
        \end{equation*}
        Since $L\psi\ne 0$, we see that the reduced row echelon form of \eqref{eqn:Jacobian} has no zero row. That is, the rows of \eqref{eqn:Jacobian} are linearly independent.
    \end{enumerate}
\end{proof}

In this paper, we are only interested in the case where the manifold of revolution is a surface, that is, $m=k$. We now state the following generalisation of Theorem \ref{thm:main_revolution}.
\begin{thm}\label{thm:main_general_revolution}
    Let $\psi:[-1,1]^l\to \R$ be a revolution function as in Definition \ref{defn:general_psi}. Consider the surface $M$ parametrised by either one of the following forms:
     \begin{equation}
        \text{(Type I) }\quad\{(s_1,s_2,r(s_1,s_2)t,r(s_1,s_2)\psi(t)):(s_1,s_2)\in [-1,1]^2,t\in [-1,1]^l\},
    \end{equation}
    where $r:[-1,1]^2\to [1,2]$ is a smooth function, or
    \begin{equation}\label{eqn:TypeII_sec2}
        \text{(Type II) }\quad \{(s_1,x(s_1,r),rt,r\psi(t)):s_1\in [-1,1],r\in [1,2],t\in [-1,1]^l\},
    \end{equation}
    where $x:[-1,1]^2\to \R$ is a smooth function. Then for every $0<\delta\ll_M 1$, $\eps\in (0,1]$ and $2\le p\le \frac{2(4+l)}{2+l}$, $N_\delta(M)$ can be $\ell^p(L^p)$ decoupled into a family $\mathcal P_{M,\eps,\delta}$ of parallelograms of width at least $\delta$ in the sense of Definition \ref{defn:decoupling} at cost $C_{\eps} \delta^{-\eps}$ , where $C_\eps$ and the overlap function of $\mathcal P_{M,\eps,\delta}$ depend only on $M,\eps$. 
    
    Moreover, if $M$ is convex and of Type I, then this $\ell^p(L^p)$ decoupling can be upgraded to $\ell^2(L^p)$ decoupling.
\end{thm}

\subsection{Gaussian curvature of surfaces of revolution}
We now compute the Gaussian curvature of the surface of revolution parametrised by
\begin{equation*}
    X=(x_1(s),\dots,x_k(s), r(s)t_1,\dots, r(s)t_l,r(s)\psi(t)),
\end{equation*}
where $s=(s_1,\dots,s_k)$ and the Jacobian matrix $(\frac {\partial (x_i(s),r(s))}{\partial s_j})_{1\le i\le k,1\le j\le k}$ has rank $k$, and $\psi$ is a revolution function as defined in Definition \ref{defn:general_psi}.

Similarly to the analysis of Proposition \ref{prop_surface_condition}, we need to consider two cases.

\subsubsection{Type I}
This case is when the Jacobian matrix $(\frac {\partial x_i(s)}{\partial s_j})_{1\le i\le k,1\le j\le k}$ is invertible. By the inverse function theorem, we may assume $x_i(s)=s_i$ for $1\le i\le k$, so $X=(s,r(s)t,r(s)\psi(t))$.

\begin{prop}\label{prop:curvature_first_case}
    The Gaussian curvature of the surface parametrised by $X=(s,r(s)t,r(s)\psi(t))$  has absolute value comparable to
\begin{equation*}
    |\det D^2 r(s)|.
\end{equation*}
In addition, if $\psi$ is concave, then the surface $X$ is convex if and only if $r$ is concave.
\end{prop}

\begin{proof}
    By direct computation, a normal vector to the surface can be given by
    \begin{equation*}
        N:=(\nabla r(s)L\psi(t),\nabla \psi(t),-1).
    \end{equation*}    
    Thus, the coefficient matrix of the second fundamental forms is given by
    \begin{equation}\label{eqn:Apr_21_01}
        |N|^{-k-l}\begin{bmatrix}
        L\psi(t) D^2 r(s) & O\\
        O & r(s) D^2 \psi(t)
    \end{bmatrix},
    \end{equation}
    where $D^2$ denotes the Hessian matrix. Since $r(s)\sim 1$ and $|N|\sim 1$, the Gaussian curvature of the surface has absolute value comparable to 
    \begin{equation*}
        |L\psi(t)\det D^2 r(s)\det D^2 \psi(t)|.
    \end{equation*} 
    Since we assumed $L\psi(t)>0$ and $\det D^2 \psi(t)\ne 0$, we see that the first sentence of this proposition follows. The statement about convexity follows from \eqref{eqn:Apr_21_01}.
\end{proof}
We also remark that in this case, considering the $\delta$-neighbourhood of the surface is equivalent to considering the $\delta$-neighbourhood of the last coordinate $r(s)\psi(t)$ only.

\subsubsection{Type II}
By a trivial partition and relabelling, we assume that the Jacobian matrix 
\begin{equation*}
    \left(\frac {\partial (x_i(s),r(s))}{\partial s_j}\right)_{1\le i\le k-1,1\le j\le k}
\end{equation*}
is invertible. By the implicit function theorem and changing notation, we may redefine the parametrisation to be
\begin{equation}
    X=(s,x(s,r),rt,r\psi(t)),
\end{equation}
where $s\in [-1,1]^{k-1}$, $r\in [1,2]$, $t\in [-1,1]^l$.

\begin{prop}\label{prop_curvature_second_case}
    The Gaussian curvature of the surface parametrised by $X=(s,x(s,r),rt,r\psi(t))$ has its absolute value comparable to
\begin{equation}
    |\partial_r x(s,r)|^l |\det D^2 x(s,r)|.
\end{equation}
In addition, if $\psi$ is concave, then the surface $X$ is convex if and only if either
\begin{itemize}
    \item $x$ is convex and $\partial_r x(s,r)$ is nonnegative, or
    \item $x$ is concave and $\partial_r x(s,r)$ is nonpositive.
\end{itemize}
\end{prop}
\begin{proof}
    By direct computation, a normal vector to the surface can be given by
\begin{equation*}
    N:=(-L\psi(t)\nabla_s x,L\psi(t), \partial_r x\nabla \psi(t),-\partial_r x).
\end{equation*}
Thus, the coefficient matrix of the second fundamental forms is given by
\begin{equation}\label{eqn:Apr_21_02}
    |N|^{-k-l}\begin{bmatrix}
        L\psi(t) D^2 x(s,r)  & O\\
        O & -r \partial_r x(s,r) D^2 \psi(t)
    \end{bmatrix}.
\end{equation}
Since $r\sim 1$, $L\psi(t)> 0$ and $\det D^2 \psi(t)\ne 0$, the first sentence of this proposition follows. The statement about convexity follows from \eqref{eqn:Apr_21_02}.

\end{proof}

We also remark that in this case, considering the $\delta$-neighbourhood of the surface is equivalent to considering the $\delta$-neighbourhood of the coordinate $x(s,r)$ only. This is because the Jacobian matrix $\frac{\partial(s,rt,r\psi(t))}{\partial (s,t)}$ is invertible.

\section{Main tools}\label{sec_tools}
In this section, we invoke the fundamental theorems established in \cite{BD2015,BD2017,Yang2,LiYang2023}, together with the two principles of decoupling established in \cite{LiYang2024}. They will be the main tools we will use to prove the new decoupling theorems.

\subsection{Bourgain-Demeter decoupling}
First, we state the celebrated decoupling theorem of Bourgain and Demeter \cite{BD2015,BD2017},  which serves as the most fundamental ingredient in this article.  To achieve our goals, we actually need the following variant, whose proof can be easily adapted from that of \cite[Proposition 5.21]{LiYang2024}.
\begin{thm}\label{thm:Bourgain_Demeter}
Let $K\ge 1$, $\zeta\in (0,1]$ and $\phi(x):[-1,1]^{k}\to \R$ be a $C^{2,\zeta}$ function with $\inf |\det D^2 \phi|\ge K^{-1}$. Let $\psi(y):[-1,1]^l\to \R$ be a $C^{2}$ function with $\inf |\det D^2 \psi|>0$.  Let $n=k+l+1$. For every $0<\delta<1$, denote by $\mathcal T_\delta$ a tiling of $[-1,1]^{n-1}$ by cubes $T$ of side length $\delta^{1/2}$. Denote $\eta(x,y)=\phi(x)+\psi(y)$.

Then for $2\le p\le \frac{2(n+1)}{n-1}$, $[-1,1]^{n-1}$ can be $\eta$-$\ell^p(L^p)$ decoupled into $\mathcal T_\delta$ at the cost of $C_\eps K^{C'_k\eps}\delta^{-\eps}$ for every $\eps>0$, in the sense of Definition \ref{defn:decoupling}, where the constant $C'_k$ depends only on $k$, and $C_\eps$ depends only on 
\begin{equation*}
k,l,\zeta,\norm{\phi}_{C^{2,\zeta}},\norm{\psi}_{C^2},\inf |\det D^2 \psi|.
\end{equation*}
Moreover, this can be upgraded to a $\ell^2(L^p)$ decoupling if $\eta(x,y)$ defines a convex surface.
\end{thm}

We now invoke the two principles of decoupling established in \cite{LiYang2024}. For completeness, we will formulate the necessary terminology and state the theorems here, and we encourage the reader to read \cite[Section 3]{LiYang2024} for more explanations and rigorous proofs of the theorems. We also rearrange the orders for better presentation.

\subsection{Rescaling invariance} The decoupling inequalities we are interested enjoy the rescaling-invariant property. In this subsection, we introduce the necessary setup and notation that facilitate this idea.

\subsubsection{Setup and notation}\label{sec:setup_notation}

Let $\mathcal U_0$ be the collection of compact sets $U\subseteq \R^{n-1}$ whose interiors are nonempty and connected. Given $0<c_0<C_0<C_1$, let $\mathcal{M}_0$ be the collection of all $C^2$ surfaces (with boundary) parametrised by $\Phi : U=U_\Phi \to \R^n$ for some compact set $U\in \mathcal U_0$, such that they satisfy the following assumptions (see notation right below for the notation $\bar \Phi$):
    \begin{enumerate} 
        \item $\|\Phi\|_{C^2(U)}\le C_0$;
        \item $\inf_U|\det D\bar\Phi|\ge c_0$;
        \item $\bar{\Phi}(U)$ contains $[-1,1]^{n-1}$ and is $C_1$-equivalent to $[-1,1]^{n-1}$.
    \end{enumerate}  
\fbox{Notation} Here and in what follows, we write $\Phi=(\bar \Phi,\Phi_n)$, where $\bar \Phi$ denotes the first $(n-1)$ coordinates of $\Phi$, and $\Phi_n$ denotes the last coordinate of $\Phi$. For $\sigma\in \R$, we denote by $\sigma \Phi$ the surface parametrised by $(\bar \Phi,\sigma \Phi_n)$. When there is no confusion, we will abuse notation and regard $\Phi$ as the surface (as a subset of $\R^n$) it parametrises. 

Unless otherwise specified, we assume that every implicit constant depends on $n,c_0,C_0,C_1$. We also write $A\approx [-1,1]^{n-1}$ to mean that $A$ is equivalent to $[-1,1]^{n-1}$.

\begin{defn}[Affine equivalence of surfaces]\label{defn:affine_equivalence_parametric}
    Let $\Phi,\Psi \in \mathcal M_0$ and let $R_\Phi, R_\Psi $ be almost parallelograms essentially contained in $[-1,1]^{n-1}$. Let $\Xi: \bar\Psi^{-1}(R_\Psi) \to \bar\Phi^{-1}(R_\Phi)$ be a diffeomorphism. 
    
    We say $(\Psi, R_\Psi)$ is $\Xi$-affine equivalent to $(\Phi,R_\Phi)$, or in symbols
    \begin{equation}
         (\Psi,R_\Psi)\overset{\Xi}\longmapsto (\Phi, R_\Phi),
    \end{equation}
    if there exists an affine bijection $\Lambda:\R^n\to \R^n$ such that (see Figure \ref{fig:cd})
    \begin{equation}\label{eqn:affine_equivalence}
        \Lambda\circ\Psi=\Phi\circ \Xi
    \end{equation}
    and that
    \begin{equation}\label{eqn:affine_equivalence_add}
        \Lambda(x,y) - \Lambda (x,0) = (0,y) \quad \forall x\in \R^{n-1}, \, y \in \R.
    \end{equation}

    Typically, $\Xi$ is clear when the pairs $(\Phi,R_\Phi)$ and $(\Psi, R_\Psi)$ are given. We sometimes simply say that $(\Phi,R_\Phi)$ and $(\Psi, R_\Psi)$ are affine equivalent.
\end{defn}
\begin{figure}[h]
    \centering
        \[
        \begin{tikzcd}[column sep = small, row sep = large]
        \bar \Psi^{-1}(R_\Psi)\arrow[ddr, out=-150, in=-180, "\Psi_n", near end, swap] \arrow{rr}{\Xi} \arrow[rightarrow]{d}{\bar\Psi} & & \bar \Phi^{-1}(R_\Phi) \arrow[rightarrow]{d}{\bar\Phi} \arrow[ddl, out=-30, in=0, "\Phi_n", near end] \\
        R_\Psi \arrow{rr}{\bar\Lambda} \arrow{dr}{\psi} & & R_\Phi \arrow{dl}[swap]{\phi}\\
        & \R  &
        \end{tikzcd}
        \]
        \caption{Diagram of affine equivalence. Here, $\bar \Lambda(x)$ is defined by the first $n-1$ coordinates of $\Lambda (x,0)$. Also, $\phi=\Phi_n\circ \bar \Phi^{-1}$ and $\psi=\Psi_n\circ \bar \Psi^{-1}$.}
        \label{fig:cd}
\end{figure}
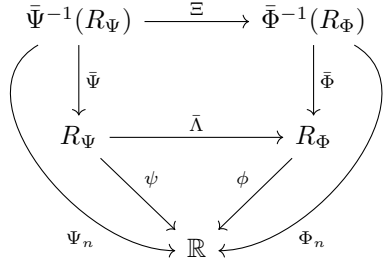

The additional Assumption \eqref{eqn:affine_equivalence_add} is to ensure that the transformed surface can still be written as a graph of the first $(n-1)$ coordinates. See \cite[Proposition 3.4]{LiYang2024}. 
Intuitively, we can think of $\Xi$ as a diffeomorphism on the ``parameter space", while $\bar \Lambda$ is its associated affine bijection on the ``coordinate space".

{\it Remark.} Strictly speaking, in \cite[Section 3]{LiYang2024} we assumed $R_\Phi$ to be exact parallelograms, not almost parallelograms. However, in view of the enlarged overlap assumption in Definition \ref{defn:enlarged_overlap}, the reader may check that the proofs go through in the same way. Therefore, we often do not distinguish between the terms ``almost parallelogram" and ``parallelogram".

\begin{defn}\label{defn:decoupling_parametric}
Let $\Phi \in \mathcal{M}_0$ and $\delta >0$.
\begin{enumerate}
    \item For $R\sub \bar \Phi(U_\Phi)$, we define the (parametric) vertical neighbourhood $N_\delta^\Phi(R)$ by
\begin{equation}\label{eqn:vertical_neighbourhood_parametric}
    N_\delta^\Phi(R) = \{ \Phi(u) + c : u \in \bar\Phi^{-1}(R) , c\in \{\vec 0\}\times [-\delta,\delta]\}.
\end{equation}
    \item For $R\sub \bar \Phi(U_\Phi)$, we say that $R \subseteq \bar \Phi(U_\Phi)$ is $(\Phi,\delta)$-flat if $N_\delta^\Phi(R)$ is contained in the $\delta$-neighbourhood of some plane given by $x_n = \alpha_0+\sum_{i=1}^{n-1}\alpha_i x_i$, for some $|\alpha_i| \lesssim 1$.
    \item We say that an almost parallelogram $R_0\sub \bar \Phi(U_\Phi)$ can be $\Phi$-$\ell^q(L^p)$ decoupled into almost parallelograms $R$, if $N_\delta^\Phi(R_0)$ can be $\ell^q(L^p)$ decoupled into $N_\delta^\Phi(R)$ (which are also almost parallelograms) as in Definition \ref{defn:decoupling}.
\end{enumerate}
\end{defn}

With the notation introduced above, we now state the following lemma on the affine invariance of decoupling.

\fbox{Notation} For the remainder of this section, we omit the decoupling exponents $p,q$ since they will be fixed.

Let $\M\sub\M_0$ be a subfamily. For each $\Phi \in \mathcal M$, let $\mathcal R_\Phi$ be a family of almost parallelograms contained in $\bar \Phi(U_\Phi)$. Let $\mathfrak R=\{\mathcal R_\Phi:\Phi\in \mathcal M\}$ be the collection of all such families $\mathcal R_\Phi$. Let $\mathcal{X}$ be a family of diffeomorphisms $\Xi$ between subsets of $\R^{n-1}$.

\begin{defn}[Rescaling invariant pair of families]\label{defn:rescaling_invariant_pair}

We say that the pair of families $(\mathcal M,\mathfrak R)$ is $\mathcal{X}$-rescaling invariant if the following holds.
\begin{enumerate}
    \item \label{item:Mar_25_rescaling_invariant_01} For every $0<\sigma\lesssim 1$, $\Phi \in \M$ and every $(\Phi,\sigma)$-flat parallelogram $R_\Phi\in \mathcal R_\Phi$, there exist $\Psi=\Psi(\Phi,R_\Phi) \in \M$ and $R_{\Psi}=R_\Psi(\Phi,R_\Phi) \in \mathcal R_{\Psi}$ equivalent to $[-1,1]^{n-1}$, and $\Xi=\Xi(\Phi,R_\Phi,\Psi,R_\Psi) \in \mathcal{X}$, such that
\begin{equation}\label{eqn:X_rescaling_invariance}
    (\sigma\Psi,R_\Psi)\overset{\Xi}{\longmapsto} (\Phi,R_\Phi)
\end{equation}
as in Definition \ref{defn:affine_equivalence_parametric}.

\item \label{item:Mar_25_rescaling_invariant_02} Given any $\Phi,\Psi\in \mathcal M$, any $R_\Phi\in \mathcal R_\Phi$, $R_\Psi\in \mathcal R_\Psi$ and any $0<\sigma\lesssim 1$ such that \eqref{eqn:X_rescaling_invariance} holds. Then for every $R_\Psi' \in \mathcal{R}_\Psi$ such that $R_\Psi' \subseteq R_\Psi$, we have $\bar\Lambda(R'_\Psi) \in \mathcal{R}_\Phi$, where $\bar \Lambda:=\bar \Phi\circ \Xi\circ \bar \Psi^{-1}$ as in Figure \ref{fig:cd}.

\end{enumerate}
\end{defn}

{\it Remark.} Note that in particular, by Part \eqref{item:Mar_25_rescaling_invariant_01} with $\sigma\sim 1$, for each $R_\Phi\in \mathcal R_\Phi$, there exist $\Phi_{R_\Phi}\in \mathcal M$, $R^0_\Phi\approx [-1,1]^{n-1}$, and $\Xi\in \mathcal X$ such that
\begin{equation*}
    (\Phi_{R_\Phi},R^0_\Phi)\overset{\Xi}{\longmapsto} (\Phi,R_\Phi).
\end{equation*}

We still need to connect the concepts of rescaling invariance and compatibility of \cite{LiYang2024} to our current setting. Let $(\mathcal M,\mathfrak R)$ be $\mathcal{X}$-rescaling invariant in the sense of Definition \ref{defn:rescaling_invariant_pair}. Define a family $\mathcal A_{\mathcal X}$ of affine bijections on $\R^{n-1}$ by
    \begin{equation}\label{eqn:defn_mathcal_A_principle}
    \begin{aligned}
        \mathcal{A}_{\mathcal X}
        = \bigcup_{0<\sigma\lesssim 1} \bigcup_{\stackrel{\Phi \in \mathcal{M}}{R_\Phi \text{ is $(\Phi,\sigma)$-flat} } } \left\{\bar \Phi\circ \Xi\circ \bar \Psi^{-1}\big| \right.&\left.\Psi=\Psi(\Phi,R_\Phi), R_{\Psi}=R_\Psi(\Phi,R_\Phi),\right.\\
        &\left.\Xi=\Xi(\Phi,R_\Phi,\Psi,R_\Psi)\right\},
    \end{aligned}        
    \end{equation}
    that is, $\mathcal A_{\mathcal X}$ consists of all affine transformations $\bar \Lambda$ associated with all mappings $\Xi$ (as in Figure \ref{fig:cd}) in Part \eqref{item:Mar_25_rescaling_invariant_01} of Definition \ref{defn:rescaling_invariant_pair}.
\begin{lem}\label{lem:mathcal_A_rescaling_compatible}
     The pair $(\mathcal M,\mathfrak R)$ is $\mathcal{A}_{\mathcal X}$-rescaling invariant and $\mathcal{A}_{\mathcal X}$-compatible as in \cite[Definition 3.5 and Definition 3.7]{LiYang2024}, respectively.
    
\end{lem}
The proof of Lemma \ref{lem:mathcal_A_rescaling_compatible} is almost immediate from the definition of $\mathcal A_{\mathcal X}$, so we leave it out.

\begin{lem}\label{lem:equivalent_decoupling_A}
Let $(\mathcal M,\mathfrak R)$ be $\mathcal X$-rescaling invariant as in Definition \ref{defn:rescaling_invariant_pair}. Let $\Phi,\Psi \in \mathcal{M}$, $\sigma\in (0,1]$ satisfy \eqref{eqn:X_rescaling_invariance}. 
For $\delta\in (0,\sigma]$, the following statements are equivalent:
   \begin{enumerate}
        \item \label{item:Mar_5_01} $R_\Phi$ can be $\Phi$-decoupled into parallelograms belonging to $\mathcal R_\Phi$ at scale $\delta$ at cost $C$.
        \item \label{item:Mar_5_02} $R_{\sigma\Psi}$ can be $\sigma\Psi$-decoupled into parallelograms belonging to $\mathcal R_{\psi}$ at scale $\delta$ at cost $C$.
        \item \label{item:Mar_5_03} $R_\Psi$ can be $\Psi$-decoupled into parallelograms belonging to $\mathcal R_{\psi}$ at scale $\sigma^{-1}\delta$ at cost $C$.
     \end{enumerate}
    
    \end{lem}

\subsection{Degeneracy locating principle}

We state and slightly rephrase the {\it degeneracy locating principle} in \cite[Section 3]{LiYang2024}.

\begin{defn}[Degeneracy determinant]\label{defn:degeneracy_determinant}
    Let $(\mathcal M,\mathfrak R)$ be an $\mathcal{X}$-rescaling invariant pair as in Definition \ref{defn:rescaling_invariant_pair}. We say that an operator $H$ on $\mathcal M$ that sends $\Phi\in \M$ to a function $H\Phi$ defined on $U_\Phi\approx [-1,1]^{n-1}$ is a degeneracy determinant, if there exist constants $C\ge 1$ and $\beta\in (0,1]$, depending only on $(\mathcal M,\mathfrak R)$, such that the following holds:

    \begin{enumerate}
        \item \label{item:Lipschitz} (Lipschitz) For every $\Phi\in \mathcal M$, 
        \begin{equation}\label{eqn:Lipschitz}
            |H\Phi(u)-H\Phi(u')|\le C|u-u'|,\quad \forall u,u'\in U_\Phi.
        \end{equation}
        
        \item \label{item:totally_degen} (Totally degenerate approximation)
        Let $\Phi \in \M$  be such that $$\norm{H \Phi}_{L^\infty(U_\Phi)}\le \sigma$$ for some $0<\sigma\lesssim 1$. Then there exists a surface $\Psi \in \mathcal{M}$ satisfying $H\Psi \equiv 0$ on $U_\Phi$ and
        \begin{equation}
            N^{\Phi}_{\sigma^{\beta}} (U_\Phi) \subseteq N^{C^{-1}\Psi}_{C\sigma^{\beta}} (U_\Phi) \subseteq N^{\Phi}_{C^2\sigma^{\beta}} (U_\Phi).
        \end{equation}
        Moreover, $\mathcal{R}_\Psi \subseteq \mathcal{R}_\Phi$.
        \item (Regularity under rescaling) Let $R \in \mathcal{R}_\Phi$, take $R_\Phi$, $\Phi_{R_\Phi}$ and $\Xi$ as in the remark right after Definition \ref{defn:rescaling_invariant_pair}. If the width of $R_\Phi$ is essentially bounded below by $\mu$, then
        \begin{equation}\label{eqn:regularity_degen_det}
            \mu^C |H\Phi (\Xi (u))| \leq |H\Phi_{R_\Phi}(u)| \leq |H\Phi(\Xi (u))|,\quad \forall \,\Xi(u)\in \bar \Phi^{-1}(R_\Phi).
        \end{equation}
    \end{enumerate}   
        
\end{defn}

\begin{defn}[Trivial covering property]\label{defn:trivial_covering}
    We say that a $\mathcal X$-rescaling invariant pair $(\mathcal M,\mathfrak R)$ satisfies the trivial covering property if there exist constants $C_1,C_2,C\ge 1$, depending only on $\mathcal M,\mathfrak R,\mathcal X$, such that the following holds for every $\Phi\in \mathcal M$: there exists a family $\mathcal R_{\mathrm{tri},\Phi}\sub \mathcal R_\Phi$ such that for every $K\ge 1$, we can cover $\bar \Phi(U_\Phi)$ by $\le CK^{C_1}$ parallelograms $R_0\in \mathcal R_{\mathrm{tri},\Phi}$, such that each $R_0$ has diameter $\le K^{-1}$ and width $\ge C^{-1}K^{-C_2}$. 
    
    Moreover, we have 
    \begin{equation*}
        (\Psi,R_\Psi)\overset{\Xi}{\longmapsto} (\Phi,R_0)
    \end{equation*}
    for some $\Psi\in \mathcal M$, $R_\Psi\approx [-1,1]^{n-1}$ and some $\Xi\in \mathcal X$ (so that its associated $\bar \Lambda\in \mathcal A_{\mathcal X}$ as in \eqref{eqn:defn_mathcal_A_principle}).
\end{defn}

\subsubsection{Statement of theorem}

We are now ready to state the main theorem of this subsection, the degeneracy locating principle adapted to our setting.

Let $(\mathcal M,\mathfrak R)$ be a $\mathcal X$-rescaling invariant pair as in Definition \ref{defn:rescaling_invariant_pair}, and assume that it satisfies the trivial covering property as in Definition \ref{defn:trivial_covering}. Let $H$ be a degeneracy determinant defined as in Definition \ref{defn:degeneracy_determinant}.

\begin{thm}[Degeneracy locating principle in parametric form]\label{thm:degeneracy_locating_principle}
    For $\Phi\in \mathcal M$, there exist subfamilies $R_{\mathrm{sub,\Phi}}$, $R_{\mathrm{deg,\Phi}}$, $R_{\mathrm{non,\Phi}}$ of $\mathcal R_\Phi$, such that the following holds for every $0<\sigma\le C_\eps^{-1}$, $\eps>0$, $\Phi\in \mathcal M$, where all constants $C$ (resp. $C_\eps$, overlap functions $B$)  below depend only on $\mathcal M,\mathfrak R,\mathcal X,H$ (resp. $\mathcal M,\mathfrak R,\mathcal X,H,\eps$):
    \begin{enumerate}[label=(\alph*)]
        \item \label{item:sublevel_set_decoupling} (Sublevel set decoupling) The subset
        \begin{equation*}
            \{\bar{\Phi}(u): u \in U_\Phi, \,|H\Phi(u)|\le \sigma\}
        \end{equation*}
        can be decoupled (as in Definition \ref{defn:decoupling}) into $B$-overlapping parallelograms $R_{\mathrm{sub},\Phi}\in \mathcal R_\Phi$ with width at least $\sigma$ at cost $C_\varepsilon\sigma^{-\varepsilon}$, on each of which we have $|H\Phi(\bar\Phi^{-1}(x))|\le C\sigma$. Moreover, we require that \begin{equation*}
        (\Psi,R_\Psi)\overset{\Xi}{\longmapsto} (\Phi,R_{\mathrm{sub}})
    \end{equation*}
    for some $\Psi\in \mathcal M$, $R_\Psi\approx [-1,1]^{n-1}$ and some $\Xi\in \mathcal X$ (so that its associated $\bar \Lambda\in \mathcal A_{\mathcal X}$ as in \eqref{eqn:defn_mathcal_A_principle}).
    
    \item \label{item:degnerate_decoupling} (Degenerate decoupling) If $H\Phi\equiv 0$ on $U_\Phi$, then $\bar\Phi(U_\Phi)$ can be $\Phi$-decoupled (as in Definition \ref{defn:decoupling_parametric}) at scale $\sigma$ into $(\Phi,\sigma)$-flat $B$-overlapping parallelograms $R_{\mathrm{deg},\Phi} \in \mathcal R_\Phi $ with width at least $\sigma$ at cost $C_\eps\sigma^{-\varepsilon}$.

    \item \label{item:nondegnerate_decoupling} (Nondegnerate decoupling) If $|H\Phi| \geq K^{-1}$ on $U_\Phi$ for some $K\ge 1$, then $\bar\Phi(U_\Phi)$ can be $\Phi$-decoupled (as in Definition \ref{defn:decoupling_parametric}) at scale $\sigma$ into $(\Phi,\sigma)$-flat $B$-overlapping parallelograms $R_{\mathrm{non},\Phi}\in \mathcal R_\Phi$ with width at least $\sigma$ 
 at cost $K^{C}C_\eps\sigma^{-\varepsilon}$. (Importantly, the power $C$ on $K$ cannot depend on $\eps$.)

\item [(*)] \label{item:lowdim} Let $\Phi\in \mathcal M$. Given a parallelogram $R$ that is either an element of $\mathcal R_{\mathrm{tri},\Phi}$ from the trivial decoupling in Definition \ref{defn:trivial_covering}, or an element of $\mathcal R_{\mathrm{sub},\Phi}$ in Assumption \ref{item:sublevel_set_decoupling}. Denote by $ w$ the width of $R$, and let $L\sub \R^{n-2}$ be the lower dimensional parallelogram obtained from $R$ by removing the shortest dimension of $R$. Then the following conditions hold:
            \begin{enumerate}[label=(\roman*)]
                \item $L$ can be $\phi|_L$-decoupled into $(\phi|_L, w)$-flat parallelograms $R'_{\mathrm{low}}$ at scale $ w$ at cost $C_\eps  w^{-\eps}$, with $R'_{\mathrm{low}}$ having width at least $ w$.
                \item Denote by $R_{\mathrm{low}}$ the parallelogram obtained from $R'_{\mathrm{low}}$ by adding back the dimension of length $  w$. Then we require that $\lambda_{R_{\mathrm{low}}}\in \mathcal A$, $R_{\mathrm{low}}\in \mathcal R_\Phi$, and that $R_{\mathrm{low}}\sub 2R$.
           
            \end{enumerate}  
    \end{enumerate}

    Then for each $\Phi\in \mathcal M$ and each $0<\delta\le C_\eps^{-1}$, $\bar\Phi(U_\Phi)$ can be $\Phi$-decoupled into $(\Phi,\delta)$-flat $B'$-overlapping parallelograms $R_{\mathrm{final}}\in \mathcal R_\Phi$ with width at least $\delta$ at cost $C'_\eps \delta^{-\eps}$, where $C'_\eps$ and $B'$ depend only on $\mathcal M,\mathfrak R,\mathcal X,H,\eps,p,q$.

\end{thm}

In order to use Theorem \ref{thm:degeneracy_locating_principle} to prove decoupling for Type II surfaces of revolution as in Theorem \ref{thm:main_general_revolution}, we now present an outline of what we need to prove in Sections \ref{sec:family_surfaces} through \ref{sec:nondegenerate_case}:
\begin{itemize} 
    \item Section \ref{sec:family_surfaces}: definition of families $\mathcal M$, $\mathfrak R$ and $\mathcal X$.
    \item Section \ref{sec:rescaling}: rescaling invariance as in Definition \ref{defn:rescaling_invariant_pair}, and the trivial covering property as in Definition \ref{defn:trivial_covering}
    \item Section \ref{sec:degeneracy_determinant}: conditions of degeneracy determinant as in Definition \ref{defn:degeneracy_determinant}
    \item Section \ref{sec:sublevel_set}: sublevel set decoupling as in \ref{item:sublevel_set_decoupling} of Theorem \ref{thm:degeneracy_locating_principle}
    \item Section \ref{sec:totally_degenerate_case}: totally degenerate decoupling and lower dimensional decoupling as in \ref{item:degnerate_decoupling} and (*) of Theorem \ref{thm:degeneracy_locating_principle}
    \item Section \ref{sec:nondegenerate_case}: nondegenerate decoupling as in \ref{item:nondegnerate_decoupling} of Theorem \ref{thm:degeneracy_locating_principle}
\end{itemize}

\subsection{Radial principle}

The second principle we state in this section is the radial principle, formulated in \cite[Section 2]{LiYang2024}. We state a slightly different but simplified version here.

\subsubsection{Setup}

For each $j\in \N$, let $\mathcal{R}_j$ be the family of all parallelograms contained in $[-3,3]^{j}$, and let $\mathcal{X}_j$ be the family of all affine bijections on $\R^j$. For $k,l\in \N$, let $\mathcal{M}$ be a family of graphs $M$ of the form
\begin{equation*}
    M=\{(s,t,\phi_1(s)+\phi_2(t)):(s,t)\in [-3,3]^{k+l}\}
\end{equation*}
where $\phi_1,\phi_2$ are $C^2$. Abusing notation, we also write $\phi_1(s)+\phi_2(t)\in \mathcal M$.

Let $\mathcal{R}$ be the family of parallelograms of the form $R_1\times R_2$ where $R_1\in \mathcal R_k$ and $R_2\in \mathcal R_l$. Let $\mathcal X$ be the family of all affine bijections $\Xi$ of the form $\Xi(s,t)=\Xi_1(s) \circ \Xi_2(t)$, where $\Xi_1\in \mathcal X_k$ and $\Xi_2\in \mathcal X_l$. 

Suppose that $(\mathcal{M},\{\mathcal{R}\})$ is an $\mathcal{X}$-rescaling invariant family as in Definition \ref{defn:rescaling_invariant_pair} (with $\Phi_n=\phi_1(s)+\phi_2(t)$), i.e., for every $0<\sigma \lesssim 1$, $ \phi_1(s)+\phi_2(t) \in \mathcal{M}$ and every $(\phi_1(s)+\phi_2(t),\sigma)$-flat parallelogram $R_1\times R_2 \in \mathcal R$, we have
    \begin{equation*}
        \sigma^{-1}(\phi_1\circ \lambda_{R_1}(s) + \phi_2 \circ \lambda_{R_2}(t) -\text{affine terms} )\in \mathcal{M},
    \end{equation*}
    where the affine terms are given by
    $$
    \phi_1\circ \lambda_{R_1}(0) + \phi_2 \circ \lambda_{R_2}(0)+\nabla_s (\phi_1\circ \lambda_{R_1})(0)\cdot s+\nabla_t (\phi_2\circ \lambda_{R_2})(0)\cdot t.
    $$

We now state the required decoupling condition (cf. \cite[Condition 2.3]{LiYang2024}).
\begin{condition}[Decoupling for sum surface]\label{condition:decoupling}  For every $\eps>0$ and every $0<\delta\ll_\mathcal M 1$, the following is satisfied for all $\phi_1(s)+\phi_2(t)  \in \mathcal{M}$: \\
$[-1,1]^{k+l}$ can be $(\phi_1(s)+\phi_2(t))$-decoupled into B-overlapping $(\phi,\delta)$-flat parallelograms in $\mathcal{R}$ of width bounded below by $\delta$, at the cost of $C_\varepsilon\delta^{-\varepsilon}$, where $B$ and $C_\varepsilon$ may only depend on $\mathcal M,p,q,\eps$.

\end{condition}

\subsubsection{Product surface}
We need a little extra regularity assumptions on the functions $r(s)$ and $\psi(t)$. Suppose that  $\psi(t)$ is polynomial-like in the sense of \cite[Definition 2.1]{LiYang2024} (which we restate below):
\begin{defn}\label{condition:2T_flat}
    We say $\psi$ is polynomial-like with parameter $C_{\mathrm{poly}}\ge 1$ if for every parallelogram $T\sub [-3,3]^l$ such that $2T\sub [-3,3]^l$, we have $\norm{\psi}_{C^2(2T)}\le C_{\mathrm{poly}} \norm{\psi}_{C^2(T)}$.
\end{defn}
For example, if $\psi$ is a nondegenerate function as in Definition \ref{defn:general_psi}, then $\psi$ is polynomial-like, by a remark after \cite[Corollary 2.1]{LiYang2024}.

We further assume that either of the following assumptions hold:
\begin{itemize}
    \item $r(s)$ is affine, or
    \item $L\psi(t)=\psi(t)-t\cdot \nabla \psi(t)\ne 0$.
\end{itemize}

\fbox{Notation} Let $r(s)$ and $\psi(t)$ be given as above. Every implicit constant below in this subsection is allowed to depend on $p,q,l,C_{\mathrm{poly}},\norm{r}_{C^2},\norm{\psi}_{C^2}$ and $\inf_{t\in [-3,3]^l} |L\psi(t)|$.\\

We are now ready to state the radial principle of decoupling. For each $\delta>0$, we define the $\delta$-neighbourhood of the ``product" surface:
\begin{align}
    \Pi_\delta(S,T)&:=\{(s,r(s)t,u):s\in S,t\in T,|u-r(s)\psi(t)|\le \delta\}.\label{eqn:defn_Pi(S,T)}
\end{align}

\begin{thm}[Radial principle of decoupling]\label{thm:radial_principle}
    Assume that Condition \ref{condition:decoupling} holds for some $\mathcal{M}$ that contains the graph of $r(s)+\psi(t)$. Then for$\eps\in (0,1)$ and $0<\delta\ll 1$, there exist families $\mathcal{S}_\delta$ and $\mathcal{T}_\delta$ of parallelograms contained in $[-2,2]^k$, $[-2,2]^l$, respectively, such that the set $\Pi_\delta([-1,1]^k,[-1,1]^l)$ can be decoupled into the following family
    \begin{equation}\label{eqn:defn_mathcal_Pdelta}
        \mathcal P_\delta:=\{\Pi_{\delta}(S,T):S\in \mathcal S_\delta,\,\,T\in \mathcal T_\delta\}
    \end{equation}
    of almost-rectangles with width at least $\delta$ at cost $\le C'_\eps\delta^{-\eps}$. Here, the constant $C'_\eps$ and the overlap functions of $\mathcal S_\delta$, $\mathcal T_\delta$ and $\mathcal P_\delta$ may only depend on $\mathcal M,p,q,\eps$.

Moreover, we have the following special case: If $\det D^2 \psi \neq 0$, then $\mathcal T_\delta$ can be taken to be a tiling of $[-1,1]^l$ by cubes of side length $\delta^{1/2}$.
\end{thm}

\section{Decoupling for Type I surfaces of revolution}\label{sec:typeI}

In this section, we prove Theorem \ref{thm:main_general_revolution} for Type I surfaces, namely, decoupling for the surface parametrised by
\begin{equation*}
    X=(s,r(s)t,r(s)\psi(t)),
\end{equation*}
where $s\in [-1,1]^2$, $r$ is an arbitrary smooth function taking values in $[1,2]$, $t\in [-1,1]^l$ and $\psi\in C^2$ is a nondegenerate function as in Definition \ref{defn:general_psi}. Thus, $\psi$ is polynomial-like as in Definition \ref{condition:2T_flat}. Also, by Corollary \ref{cor:decoupling_condition_bivariate}, we can check that Condition \ref{condition:decoupling} holds. Applying Theorem \ref{thm:radial_principle} finishes our task.

\section{Pseudo-polynomials}\label{sec:pseudo-polynomials}

To study the decoupling of the much more difficult Type II surfaces, we need to first introduce the notion of pseudo-polynomials. In simple words, pseudo-polynomials are univariate functions arising from zeros of bivariate polynomials satisfying certain scale assumptions.

Let $\Omega_0\sub \R^2$ be a parallelogram of the form
    \begin{equation}\label{eqn:Omega_0}
        \Omega_0=\{(s_1,s_2)\in \R^2:s_1\in [-1,1],|s_2-ms_1-b|< h\},
    \end{equation}
where $m,b\in [-1,1]$ and $h\in (0,1]$. 

Let $P\in \mathcal P_{2,d}$. Assume for some $\sigma_1\lesssim \sigma_2\lesssim 1$ that
\begin{equation} \label{eqn:psudo_poly_condition}
    |\partial_1 P(s)|\lesssim \sigma_1, \quad |\partial_2 P(s)|\sim \sigma_2,  \quad \forall s \in \Omega_0.
\end{equation}
Here and throughout this section, the implicit constants depend on $d$ only.

\begin{prop} \label{prop:exist_g}
    Let $Y$ be the range of of the polynomial $P$ over $\Omega_0$. Then $Y$ is an open interval, and for each $y\in Y$, the set $P^{-1}(y)\cap \Omega_0$ is the union of $O_d(1)$ many graphs of algebraic functions $g_i(s_1)$ over a subinterval $I_{i}\sub [-1,1]$. Moreover, $|g'_i(s_1)|\lesssim_d \frac{\sigma_1}{\sigma_2}\lesssim 1$.
\end{prop}
\begin{proof}[Proof of Proposition \ref{prop:exist_g}]
    We may assume $\partial_2 P>0$. For each $s_1$, the set $\{P(s_1,s_2):|s_2-ms_1-b|<h\}$ is an open interval $I(s_1)$ given by $(P(s_1,ms_1+b-h, P(s_1,ms_1+b+h))$. The range $Y$ is thus the union of $I(s_1)$ over $s_1\in [-1,1]$, and so $I$ is an open interval. For every $y\in Y$, by the fundamental theorem of algebra, the set $P^{-1}(y)\cap \Omega_0$ is the union of $O_d(1)$ many algebraic curves, and by the implicit function theorem, each one is the graph of some function $g(s_1)$ over a subinterval $I_i\sub [-1,1]$ with $|g'_i(s_1)|\lesssim \frac{\sigma_1}{\sigma_2}\lesssim 1$.
\end{proof}

For decoupling consideration, we may of course locate to one such interval $I_i$. For simplicity, we may abuse notation and pretend there is only one such $i$, and denote $g=g_i$, $I_0=I_{0,i}$.

\begin{defn}\label{eqn:defn_pseudo_polynomials}
    Such function $g:I_0\to \R$ obeying $P(s_1,g(s_1))=y$ will be referred to as a (univariate) {\it pseudo-polynomial} of degree (at most) $d$.
\end{defn}

We record a few propositions below for future use. In all the following, the degree $d$ of the pseudo-polynomial will not change.
\begin{prop}\label{prop:rescaling_pseudo}
    If $g$ is a pseudo-polynomial over $I_0$, and $\lambda:[-1,1]\to I_0$ is an affine bijection, then $f:=g\circ \lambda$ is also a pseudo-polynomial.
\end{prop}
\begin{proof}
    By assumption, let $P(s_1,g(s_1))=y$ for $s_1\in I_0$, such that \eqref{eqn:psudo_poly_condition} holds. Define $Q(s_1,s_2):=P(\lambda s_1,s_2)$, so that $Q(s_1,f(s_1))=y$. Note that we still have $|\partial_1 Q(s)|\lesssim \sigma_1\lesssim \sigma_2\sim |\partial_2 P(s)|$.
\end{proof}

\begin{prop}\label{prop:fta_pseudo}
    For each pseudo-polynomial $g$ of degree at most $d$ and each real number $c$, the equation $g(s_1)=c$ either always holds in $I_0$, or it has at most $d$ solutions in $I_0$. Moreover, the equations $g'(s_1)=c$ and $g''(s_1)=c$ either always hold, or they only have $O_d(1)$ many solutions.
\end{prop}
\begin{proof}
For each $s_1\in I_0$, $g(s_1)=c$ means $h(s_1):=P(s_1,c)=y$ for some $y\in \R$. If the polynomial $h$ is constant, this means $h'(s_1)=\partial_1 P(s_1,c)\equiv 0$, which is a contradiction to $|\partial_1 P|\sim \sigma_1$. Thus, $h$ cannot be constant, and so the result follows from the fundamental theorem of algebra.

We only give the proof of the assertion about $g''$ since the assertion about $g'$ is similar and easier. A direct computation shows that
    \begin{equation}\label{eqn_F1F3_pseudo}
    g''(s_1) = - \frac{P_{11} P_2^2 - 2 P_1 P_2 P_{12} + P_1^2 P_{22}}{P_2^3} (s_1,g(s_1)).
    \end{equation}   

    Therefore, $g''(s_1) -c = 0$ if and only if
    $$
    \begin{cases}
        (P_{11} P_2^2 - 2 P_1 P_2 P_{12} - P_1^2 P_{22} + c P^3_2) (s_1,s_2) = 0, \\
        P(s_1,s_2) = 0.
    \end{cases}
    $$
    Therefore, the number of solutions of $g''=c$ is exactly the number of intersection of the two algebraic curves of degree $d$ and $3d$, respectively. 
    
    If the two algebraic curves have common components and the zero set of one of these components has nonempty intersection with the interior of the square $[-1,1]^2$, there exists an open interval where $g'' = c$. 

    If the two algebraic curves have no common components, then by Bezout's theorem, there are at most $3d^2$ many zeros.
\end{proof}

\begin{prop}\label{prop:I_2} 
$g''$ and all its derivatives are bounded above by $O_d(1)$.
\end{prop}

\begin{proof}
Let $\lambda:[-1,1]\to I_0$ be an affine bijection, and denote $f:=g\circ \lambda$. We then need to show that for each $k\ge 2$, $f^{(k)}$ is bounded above by $O_d(|I_0|^k)$.

Let $h$ be the smallest positive number such that
\begin{align}
    &\{(s_1,s_2):s_1\in I_0,|s_2-f(s_1)|<\delta\}\nonumber\\
    \sub \,\,& \Omega:=\{(s_1,s_2):s_1\in I_0,|s_2-f(0)-f'(0)s_1|\le h\},\label{eqn_defn_B}
\end{align}
so that in $\Omega$ we have $|\partial_2 P|\sim \sigma_2$. Let 
     $$
     T_\Omega(s_1,s_2) = (\lambda(s_1), h s_2 + f'(0)s_1+f(0))
     $$ be an affine map that sends $[-1,1]^2$ to $\Omega$. Let 
     $$
     \Tilde{P}(s_1,s_2) : = (h \sigma_2)^{-1} (P\circ T_\Omega) = (h \sigma_2)^{-1} P (\lambda(s_1), h s_2 +f(0)+  f'(0)s_1)
     $$
     and 
     $$
     \tilde f (s_1) : = h^{-1} (f(s_1)-f(0)-f'(0)s_1)
     $$
     be such that $\Tilde P(s_1, \tilde f(s_1)) =y$.
     Since in $\Omega$ we have
     $$
    |P(s_1,s_2)| = |P(s_1,s_2) - P(s_1,f(s_1))| \leq h\sup_{\Omega}|\partial_2 P| \sim  h\sigma_2,
    $$
    we see that $\tilde P\in \mathcal P_{2,d}$. Moreover, direct computation shows that $|\partial_2 \tilde P| \sim 1$ on $[-1,1]^2$. More generally, we have for each $j,k\ge 0$ that
    \begin{equation*}
        |\partial_1^k \partial_2^j\tilde P|\lesssim |I_0|^k.
    \end{equation*}    
    Hence the implicit function theorem implies that for each $k\ge 1$, 
    \begin{equation*}
        |\tilde f^{(k)}|\lesssim |I_0|^k.
    \end{equation*}    
\end{proof}
The proof above also gives the following corollary.
\begin{cor}\label{cor:I_2}
    If $g:[-1,1]\to \R$ is a pseudo-polynomial centred at $c$ such that $|g''|\sim \sigma$, then the function
    \begin{equation*}
        f(s):=\sigma^{-1}(g(s)-g(0)-g'(c)s)
    \end{equation*}
    is also a pseudo-polynomial, and $|f''|\sim 1$.
\end{cor}

The following proposition is similar to Proposition \ref{prop:I_2}.
\begin{prop}\label{prop:g'_sim_sigma}
    If $g$ is a pseudo-polynomial over $I_0=[-1,1]$, and $|g'|\sim \sigma$, then the function
    \begin{equation}
        f(s):=\sigma^{-1}(g(s)-g(0))
    \end{equation}
    is also a pseudo-polynomial over $I_0=[-1,1]$ with $|f'|\sim 1$.
\end{prop}
\begin{proof}
    Let $P\in \mathcal P_{2,d}$ such that $P(s,g(s))=y$, and write $|\partial_1 P|\lesssim \sigma_1$, $|\partial_2 P|\sim \sigma_2$. Then we set $\tilde P(s_1,s_2):=c_d P(s_1,\sigma s_2+g(0))$, which lies in $\mathcal P_{2,d}$ for some suitable $c_d\sim 1$. One can check that $\tilde P(s,f(s))=0$, and $|\partial_1\tilde P|\lesssim \sigma_1$, $|\partial_2 P|\sim \sigma_2 \sigma$. But by construction, $|f'|\sim 1$, and so by the implicit function theorem, we must have $\sigma_1\sim \sigma_2\sigma$ (thus $|\partial_1 \tilde P|\sim \sigma_1$). Thus, $f$ is a pseudo-polynomial over $[-1,1]$.
\end{proof}

\begin{prop}\label{prop:inverse_pseudo}
    If a pseudo-polynomial $g$ over $I_0=[-1,1]$ satisfies $|g'|\sim 1$, then the inverse function $h$ of $g$ is also a pseudo-polynomial.
\end{prop}

\begin{proof}
    Let $P(s_1,g(s_1))=0$, and define $Q(s_1,s_2):=P(s_2,s_1)$. Then $Q(s_2,h(s_2))=P(h(s_2),s_2)=P(s_1,g(s_1))=0$. Moreover, since $|g'|\sim 1$, in the domain $\Omega_0$ in \eqref{eqn:Omega_0} we must have $\max\{|m|,h\}\sim 1$. This means that the reflection of $\Omega_0$ about the straight line $s_2=s_1$ is also essentially a parallelogram $\Omega_0'$ of the same form as in \eqref{eqn:Omega_0}. Also, in this case, $\sigma_2\sim \sigma_1$, and $|\partial_1 Q|\sim |\partial_2 Q|\sim \sigma_1$ in $\Omega_0'$.
\end{proof}

\section{Family of surfaces}\label{sec:family_surfaces}
In this section, we will embed Type II surfaces of revolution into a large family that enjoys rescaling invariance. 

Recall from \eqref{eqn:TypeII_sec2} that Type II surfaces are given by 
\begin{equation}\label{eqn:TypeII_before_reduction}
        \{(s_1,x(s_1,r),rt,r\psi(t)):s_1\in [-1,1],r\in [1,2],t\in [-1,1]^l\}.
\end{equation}

\subsection{Preliminary reductions}
For technical reasons, we start with some reductions first. First, for simplicity of notation later, we swap the second coordinate $x(s_1,r)$ with the last coordinate. Also, since $r\in [1,2]$, we may equivalently write \eqref{eqn:TypeII_before_reduction} into
\begin{equation}\label{eqn:Mar_31_01}
    \{(s_1,(1+s_2/2)\psi(t),(1+s_2/2)t,x(s)):s\in [-1,1]^2,t\in [-1,1]^l\}.
\end{equation}
The next step is to partition $[-1,1]^{2+l}$ into smaller cubes of side length $c$, where $c=c(\psi)$ is a small enough constant to be determined. We can apply a trivial decoupling to locate to each such cube $S_0\times T_0$. For the $s$ variable, we can assume without loss of generality that $S_0$ is centred at $0$, and then rescale $S_0$ back to $[-1,1]^2$. Thus, renaming $x$, we can equivalently write \eqref{eqn:Mar_31_01} into
\begin{equation}\label{eqn:Mar_31_02}
    \{(s_1,(1+cs_2)\psi(t),(1+cs_2)t,x(s)):s\in [-1,1]^2,t\in T_0\}.
\end{equation}
Denote by $t_0$ the centre of $T_0$. By a translation $t\mapsto t+t_0$ and the change of variables
\begin{equation*}
    -Q_0(t):=\psi(t+t_0)-\psi(t_0)-\nabla \psi(t_0)\cdot t,
\end{equation*}
we can write \eqref{eqn:Mar_31_02} into
\begin{equation*}
    \{(s_1,(1+cs_2)(-Q_0(t)+\psi(t_0)+\nabla \psi(t_0)\cdot t,(1+cs_2)(t+t_0),x(s)):s\in [-1,1]^2,t\in [-c,c]^l\},
\end{equation*}
where $Q_0(0)=0$, $\nabla Q(0)=0$ and $\det D^2 Q\ne 0$. We then apply a few affine equivalences to all but the last coordinate:
\begin{equation*}
\begin{aligned}
    &(s_1,(1+cs_2)(-Q_0(t)+L\psi(t_0)),(1+cs_2)(t+t_0))\\
    \iff &\left(s_1,(1+cs_2)(-Q_0(t)+L\psi(t_0)),(1+cs_2)\left(t+\frac{t_0Q_0(t)}{L\psi(t_0)} \right)\right)\\
    \iff &\left(s_1,-(1+cs_2)Q_0(t)+cL\psi(t_0)s_2,(1+cs_2)\left(t+\frac{t_0Q_0(t)}{L\psi(t_0)} \right)\right),
\end{aligned}    
\end{equation*}
where $s\in [-1,1]^2$, $t\in [-c,c]^l$. Rescaling $t$ back to $[-1,1]^l$ gives
\begin{equation*}
    \left(s_1,-(1+cs_2)Q_0(ct)+cL\psi(t_0)s_2,(1+cs_2)\left(ct+\frac{t_0Q_0(ct)}{L\psi(t_0)} \right)\right),
\end{equation*}
which is further equivalent to
\begin{equation*}
    \left(s_1,s_2-(1+cs_2)Q(t),(1+cs_2)\left(t+t_0 Q(t) \right)\right),
\end{equation*}
where we have set 
\begin{equation}\label{eqn:defn_Q_from_Q0}
    Q(t):=\frac{Q_0(ct)}{cL\psi(t_0)},
\end{equation}
so that $\norm{Q}_{C^2}\sim c$.

Lastly, by Taylor expansion as in \cite[Section 2]{LiYang2023}, we may assume $x\in \mathcal P_{2,d}$. 
It then suffices to prove decoupling for the following surface
\begin{equation}\label{eqn:Mar_31_03}
    \{\left(s_1,s_2-(1+cs_2)Q(t),(1+cs_2)\left(t+t_0 Q(t) \right),x(s)\right):(s,t)\in [-1,1]^{2+l}\},
\end{equation}
with the decoupling constants independent of the coefficients of $x$.

\subsection{Families of surfaces and almost parallelograms}
We now proceed to define a family $\mathcal M$ of surfaces so that \eqref{eqn:Mar_31_03} embeds into $\mathcal M$.

\subsubsection{Family $\mathcal Q$}
Fix $c\sim 1$ depending only on $\psi$. For suitable constants $c_1,c_2$ depending on $c$, we define a family $\mathcal Q=\mathcal Q(c_1,c_2)$ of functions satisfying
\begin{equation}\label{eqn:Q_small}
    \begin{aligned}
        &Q(0)=0,\,\nabla Q(0)=0,\\        
&\norm{Q(t)}_{C^2([-1,1]^l)}\le c_2,\\
        &\inf_{t\in [-1,1]^l}|\det D^2 Q(t)|\ge c_1,
    \end{aligned}
\end{equation}
such that the function $Q$ as in \eqref{eqn:defn_Q_from_Q0} lies in $\mathcal Q$.

\fbox{Notation} In the remainder of this section, we omit the dependence of constants on $d,l$ (as well as $c,c_1,c_2$).

\subsubsection{Family $\mathcal A$}

For $a=(a_0,a_1,a_2)\in \R^3$ and $s=(s_1,s_2)\in [-1,1]^2$, denote
\begin{equation*}
    A_a(s):=a_0+a_1s_1+a_2s_2,
\end{equation*}
and with this, we denote
\begin{equation}\label{eqn:family_A}
    \mathcal A:=\{A=A_a: (a_0,a_1,a_2)\in \R^3, \max_{s\in [-1,1]^2}|A_a(s)-1|\le c\}.
\end{equation}
That is, $\mathcal A$ consists of affine functions $A$ mapping $[-1,1]^2$ into $[1-c,1+c]$.

\subsubsection{Parametrisation of coordinate space}

For $A\in \mathcal A$, $Q\in \mathcal Q$ and $u\in [-1,1]^l$, define the function $\bar\Phi=\bar\Phi_{A,Q,u}$ by
\begin{equation}\label{eqn:Phi_diffeomorphism}
    \bar\Phi(s_1,s_2,t):=(s_1,s_2-A(s)Q(t), A(s)(t+uQ(t))).   
\end{equation}
It is then easy to see that the Jacobian matrix $D\bar \Phi$ is within $O(c)$ from the identity matrix; so if $c$ is small enough, then $D\bar \Phi$ is invertible.

\subsubsection{Family of surfaces}
Given $\bar \Phi_{A,Q,u}$. For $x\in \mathcal P_{2,d}$, we define $\Phi=\Phi_{A,Q,u,x}$ by
\begin{equation}\label{eqn:defn_M}
\Phi:=\left\{\begin{bmatrix}
    s_1\\
    s_2-A(s)Q(t)\\    
    A(s)(t+u Q(t))\\
    x(s_1,s_2)
\end{bmatrix}:(s,t)\in [-1,1]^{2+l}\right\},
\end{equation}
so that we have taken $U_\Phi=[-1,1]^{2+l}$ for all $\Phi$. We thus define
\begin{equation}\label{eqn:defn_family_M}
    \mathcal M:=\{\Phi_{A,Q,u,x}:A\in \mathcal A,Q\in \mathcal Q,u\in [-1,1]^l,x\in \mathcal P_{2,d}\}.
\end{equation}
In general, parallelograms in the coordinate space cannot be easily parametrised by $(s,t)$ also lying in a parallelogram, and hence the effects of rescaling of these parallelograms are difficult to compute. Instead, we will work on the rescaling in the parameter space and prove that this diffeomorphism induces an affine transformation for a specific family $\mathcal{R}_\Phi$ of almost parallelograms.

\subsubsection{Family $\mathcal S$ of planar parallelograms}
It is easy to see that every almost parallelogram $S\sub [-1,1]^2$ is equivalent to a vertical neighbourhood, namely, of the form \eqref{eqn:Omega_0}. We thus define the family $\mathcal S$ consisting of all images of affine transformations $\lambda$ over $[-1,1]^2$, where
\begin{equation}\label{eqn:affine_lambda_S}
    \lambda\begin{bmatrix}
        s_1 \\
        s_2
    \end{bmatrix}
    :=
    \begin{bmatrix}
        m_{11} & 0\\
        m_{21} & \eta
    \end{bmatrix}
    \begin{bmatrix}
        s_1\\ s_2
    \end{bmatrix}
    +\begin{bmatrix}
        s_{1,0} \\ s_{2,0}
    \end{bmatrix},
\end{equation}
where $m_{11}\in (0,1]$, $m_{21}\in [-1,1]$, $\eta\in (0,1]$, and $s_0=(s_{1,0},s_{2,0})\in [-1,1]^2$ is the centre of $S$. Note that all $S\in \mathcal S$ are essentially contained in $[-1,1]^2$. 

We sometimes denote $\eta=\eta_S$.

\subsubsection{Families $\mathcal R_\Phi$ of almost parallelograms}
We now define for each $\Phi\in \mathcal M$ a family $\mathcal R_\Phi$ of parallelograms associated to $\bar\Phi=\bar\Phi_{a,Q,u}$:
\begin{equation}\label{eqn:defn_family_R}
\begin{aligned}
    \mathcal R_\Phi&:=\{\bar\Phi(S\times T):S\in \mathcal S,\,
    T\sub [-1,1]^l\text{ is a square, } l(T)^2=\eta_S\}.
\end{aligned}
    \end{equation}
Note that as long as $l(T)^2\lesssim \eta_S$, $\bar\Phi(S\times T)$ is an almost parallelogram contained in $[-1,1]^{2+l}$.

We denote the family of all interesting collections $\mathcal R_\Phi$ of parallelograms by $\mathfrak R$, namely,
\begin{equation}\label{eqn:mathfrak_R}
    \mathfrak R:=\{\mathcal R_\Phi:\Phi\in \mathcal M\}.
\end{equation}

\subsubsection{Family $\mathcal X$ of diffeomorphisms}
We define a family $\mathcal X$ of affine bijections on $\R^{2+l}$ by 
\begin{equation}\label{eqn:defn:mathcal_X}
    \mathcal X:=\{\Xi:\Xi(s,t)=(\lambda_S(s),\lambda_T(t)),S\in \mathcal S,T \text{ is a square, } l(T)^2=\eta_S\}.
\end{equation}

This finishes the setup of the collections $\mathcal M,\mathfrak R,\mathcal X$. 

\section{Rescaling invariance}\label{sec:rescaling}

In this section, we prove the trivial covering property as in Definition \ref{defn:trivial_covering}, and the rescaling invariance as in Definition \ref{defn:rescaling_invariant_pair}, for the families $\mathcal M,\mathfrak R,\mathcal X$ defined in Section \ref{sec:family_surfaces}.

\fbox{Notation} In this section, we omit the dependence of constants on $d,l$.

\subsection{Trivial covering property}
We first deal with the trivial covering property as in Definition \ref{defn:trivial_covering}. Note now $k=2+l$. 

Indeed, since given any $\Phi\in \mathcal M$ and any $K\ge 1$, we may partition $[-1,1]^{2+l}$ into $\sim K^{2+l}$ many smaller squares $S$ of diameter $\ll K^{-1}$, such that each $\bar \Phi(S)$ has diameter $\le K^{-1}$. Then it is easy to see that these $\bar \Phi(S)$ cover $\bar\Phi([-1,1]^{2+l})$.

\subsection{Rescaling invariance}

Now we come to the main proposition of this section.
\begin{prop}\label{prop:rescaling_invariance}
    $(\mathcal M,\mathfrak R)$ is $\mathcal X$-rescaling invariant as in Definition \ref{defn:rescaling_invariant_pair}.
\end{prop}

To prove the main proposition, we prove a key lemma first. For future use, we need to consider a more general setting.

\subsubsection{Enlarged families}
Fix $v,w\in [-1,1]$. Slightly abusing notation of $\Phi$ and $\bar \Phi$, we consider even more general surfaces 
\begin{equation}\label{eqn:general_M}
    \Phi_{A,Q,u,x,v,w}:=\left\{(\bar \Phi(s,t),x(s)+wA(s)Q(t)):(s,t)\in [-1,1]^{2+l}\right\},
\end{equation}
where $\bar \Phi=\bar\Phi_{A,Q,u,v}$ is defined by
\begin{equation}\label{eqn:general_Phi}
    \bar\Phi(s_1,s_2,t):=(s_1,s_2-vA(s)Q(t), A(s)(t+uQ(t))).
\end{equation}
Note that \eqref{eqn:general_M} and \eqref{eqn:general_Phi} reduce to \eqref{eqn:defn_M} and \eqref{eqn:Phi_diffeomorphism}, respectively, when $v=1$ and $w=0$. Also, $D\bar\Phi_{A,Q,u,v}$ is invertible. Similar to \eqref{eqn:defn_family_R}, we also define for each more general $\Phi=\Phi_{A,Q,u,x,v,w}$ the following collection of almost parallelograms:
\begin{equation}\label{eqn:defn_family_R_leq}
\begin{aligned}
    \mathcal R'_{\Phi}:=\{\bar\Phi(S\times T):S\in \mathcal S,\,
    T\sub [-1,1]^l\text{ is a square, } |v|l(T)^2\le \eta_S\}.
\end{aligned}
    \end{equation}
Note that for $\Phi_{A,Q,u,x}=\Phi_{A,Q,u,x,0,0}$, we have $\mathcal R'_{\Phi}\supseteq \mathcal R_\Phi$ because of the $\le $ sign in \eqref{eqn:defn_family_R_leq}.

Now we introduce the following main lemma, which directly implies Proposition \ref{prop:rescaling_invariance} by taking $v=1$, $w=0$. We choose to prove this more general lemma since it will also be used in Section \ref{sec:nondegenerate_case}.

\begin{lem}\label{lem:rescaling_invariance}
    Fix $A\in \mathcal A$, $Q\in \mathcal Q$, $u\in [-1,1]^l$, $x\in \mathcal P_{2,d}$, $v,w\in [-1,1]$. Given a parallelogram $S\in \mathcal S$ and a square $T\sub [-1,1]^l$ such that $\bar\Phi_{A,Q,u,v}(S\times T)\in \mathcal R'_{\Phi}$.
    
    Let $\sigma\in (0,1]$, and assume that there exist scalars $\alpha_0,\alpha_1,\alpha_2$ and a vector $\alpha\in \R^{l}$ such that
    \begin{equation}\label{eqn:sigma_flatness_rescaling_invariance}
    \begin{aligned}
        &\sup_{s\in S, t\in T}|x(s_1,s_2)+wA(s)Q(t)\\
        &-\alpha_0-\alpha_1 s_1-\alpha_2 (s_2-vA(s)Q(t))-\alpha\cdot A(s)(t+uQ(t))|\le \sigma.
    \end{aligned}        
    \end{equation}
    Then there exist affine bijections $\Xi:[-1,1]^{2+l}\to S\times T$ and $\Lambda:\R^{3+l}\to \R^{3+l}$ satisfying
    \begin{equation}\label{eqn:Oct_08_01}
        \Lambda(s,t,y)-\Lambda(s,t,0)=(0,0,y),\quad \forall (s,t,y) \in \R^{3+l},
    \end{equation}
    and there exist $\tilde A\in \mathcal A$, $\tilde Q\in \mathcal Q$, $\tilde u\in [-1,1]^l$, $\tilde x\in \mathcal P_{2,d}$, $\tilde v,\tilde w\in [-1,1]$, such that
    \begin{equation}\label{eqn:Oct_08_02}
        \Lambda\circ  \tilde\Phi\equiv  \Phi\circ \Xi,
    \end{equation}
    where 
    \begin{equation*}
        \tilde \Phi(s,t):=(\bar\Phi_{\tilde A,\tilde Q,\tilde u,\tilde v}(s,t),\sigma\tilde x(s)+\sigma\tilde w \tilde A(s)\tilde Q(t)).
    \end{equation*}
    
    Moreover, we have the following special cases:
    \begin{itemize}
        \item If $\eta_S=l(T)^2$, $v=1$ and $w=0$, then we can guarantee that $\tilde v=1$ and $\tilde w=0$. This means that if $\bar \Phi_{A,Q,u}(S\times T)\in \mathcal R_\Phi$, then $\tilde \Phi\in \mathcal M$. (This is used in the proof of Proposition \ref{prop:rescaling_invariance}.)
        \item If $\sigma=l(T)^2$, $v=w=1$, and $\partial_2 x(s_0)=0$ (where $s_0$ is the centre of $S$), then we have $\tilde w=1$, and the following relation holds:
        \begin{equation}\label{eqn:Oct_14_10}
            \sup_{s\in [-1,1]^2}|\tilde v \partial_2 \tilde x(s)|= \sup_{s\in S}|x(s)-x(s_0)|.
        \end{equation}
        This will be used in the proof of Theorem \ref{thm:nondegenerate_decoupling_PS}.
    \end{itemize}
    
\end{lem}

\begin{proof}[Proof of Proposition \ref{prop:rescaling_invariance} assuming Lemma \ref{lem:rescaling_invariance}]
Part \eqref{item:Mar_25_rescaling_invariant_01} of Definition \ref{defn:rescaling_invariant_pair} follows directly from the first special case of Lemma \ref{lem:rescaling_invariance} mentioned above. To prove Part \eqref{item:Mar_25_rescaling_invariant_02} of Definition \ref{defn:rescaling_invariant_pair}, we need to prove that
\begin{equation*}
    \bar \Phi\circ \Xi\circ \bar \Psi^{-1}(R'_\Psi)\in \mathcal R_\Phi.
\end{equation*}
But this is immediate if one unravels the notation of \eqref{eqn:defn_family_R} and \eqref{eqn:defn:mathcal_X}.
\end{proof}

\subsection{Proof of Lemma \ref{lem:rescaling_invariance}}

\subsubsection{Rescaling in parameter space}\label{sec:rescale_ST}
Fix any $S$ and any $T$ satisfying the conditions of the lemma.

\fbox{Defining $\Xi$}

Let $\lambda:[-1,1]^2\to S$ be an affine bijection of the form \eqref{eqn:affine_lambda_S}, and define $\Xi$ by
\begin{equation*}
    \Xi (s,t):=(\lambda (s), \tau t+t_0)
\end{equation*}
where $\tau=l(T)/2$ and $t_0$ is the centre of $T$. Thus, $\Xi$ is an affine bijection from $[-1,1]^{2+l}$ to $S\times T$. We also record that $|v|\tau^2\le \eta=\eta_S$.

\fbox{Defining $\tilde A$} 

We define $\tilde A(s):=A(\lambda s)$, which lies in $\mathcal A$ (recall \eqref{eqn:family_A}).

\fbox{Defining $\tilde Q$}

We define
\begin{equation}\label{eqn:Q_tilde}
        \tilde Q(t):=\tau^{-2}\left(Q(\tau t+t_0)-Q(t_0)-\nabla Q(t_0)\cdot \tau t_0\right),
    \end{equation}
    which one can check belongs to $\mathcal Q$. Using this and affine invariance, we can rewrite the surface $\Phi$ in \eqref{eqn:general_M} as
    \begin{equation*}
    \begin{bmatrix}
        s_1\\
        \eta s_2-v\tilde A(s)\left(Q(t_0)+\nabla Q(t_0)\cdot \tau t+\tau^2 \tilde Q(t)\right )\\ \tilde A(s)\left(\tau t+t_0+u(Q(t_0)+\nabla Q(t_0)\cdot \tau t+\tau^2 \tilde Q(t) )\right)\\
        x(\lambda s)+w\tilde A(s)\left(Q(t_0)+\nabla Q(t_0)\cdot \tau t+\tau^2 \tilde Q(t)\right)
    \end{bmatrix}.
    \end{equation*}
    We refer to the above coordinates of $\Phi$ as the first, second, vector, and last coordinates, respectively.
    
    \subsubsection{Simplifying vector coordinate, Part I}
    
    We first pay attention to the vector coordinate, which can be rewritten as (regarding $t,t_0,u$ as column vectors)
    \begin{equation*}
        \tilde A(s)\left(\tau J_0^{-1} t
        +t_0+u (Q(t_0)+\tau^2 \tilde Q(t))
        \right),
    \end{equation*}
    where
    \begin{equation*}
        J_0^{-1}:=\begin{bmatrix}
            1+u_1 \partial_1 Q(t_0) & \cdots & u_1 \partial_l Q(t_0)\\
            \vdots & \ddots & \vdots\\
            u_l \partial_1 Q(t_0) & \cdots & 1+ u_l \partial_l Q(t_0)
        \end{bmatrix}
    \end{equation*}
    is a constant invertible matrix. We then apply $J_0$ to the vector coordinate to transform this to
    \begin{equation}\label{eqn:Sep_18_01}
        \tilde A(s)\left(\tau t+J_0t_0+(J_0 u) (Q(t_0)+\tau^2 \tilde Q(t))\right).
    \end{equation}
    
    \subsubsection{Simplifying second coordinate}
    
    We now apply affine invariance to eliminate the affine terms of the last coordinate. We first use \eqref{eqn:Sep_18_01} and an affine transformation to eliminate the first order terms in $t$:
    \begin{equation}\label{eqn:Sep_18_02}
        \eta s_2-v\tilde A(s)(b'+b''\tau^2 \tilde Q(t)),
    \end{equation}
    where
    \begin{align*}
        b'&:=Q(t_0)-(J_0 t_0+Q(t_0)J_0 u)\cdot\nabla Q(t_0),\\
        b''&:=1-J_0u \cdot\nabla Q(t_0),
    \end{align*}    
    which can be further rearranged as
    \begin{align*}
        &\eta s_2 - v (\tilde{a}_0 + \tilde{a}_1s_1 + \tilde{a}_2s_2(s)b') - b''\tau^2 \tilde{Q}(t) \\
        & = \eta(1-\tilde{a}_2vb') s_2 -vb''\tau^2 \tilde A(s) \tilde Q(t) - v (\tilde{a}_0 + \tilde{a}_1s_1)\\
        & = \eta(1-\tilde{a}_2vb') \left ( s_2 -\frac{v b''\tau^2 }{\eta(1-\tilde{a}_2vb')}\tilde A(s)\tilde Q(t) \right) - v (\tilde{a}_0 + \tilde{a}_1s_1).
    \end{align*}

    Thus, by an affine transformation (including rescaling), \eqref{eqn:Sep_18_02} is changed to
    $$
         s_2-\bar v \tilde A(s)\tilde Q(t),
    $$
    where $\bar{v} = \frac{v b''\tau^2 }{\eta(1-\tilde{a}_2vb')}.$ But since $|v|\tau^2\lesssim \eta$, we see that $|\tilde v|\lesssim 1$. 

    {\it Remark.} Strictly speaking, we need $|\tilde v|\le 1$ to prove rescaling invariance. This can be achieved by first applying a trivial partition to $T$ into $O(1)$ smaller cubes, and rescale each smaller cube. Since this is harmless for decoupling considerations, we will ignore these minor issues of constant multiples.

     \fbox{Defining $\tilde v$}
     
     By the remark right above, we may take $b''=1$ and $b'=0$ for simplicity. Thus, we have
     \begin{equation}\label{eqn:defn_tilde_v}
         \tilde v:=v\tau^2\eta^{-1},
     \end{equation}
     which lies in $[-1,1]$. Thus, we now have reduced to 
      \begin{equation*}
    \begin{bmatrix}
        s_1\\
        s_2-\tilde v \tilde A(s)\tilde Q(t)\\        
        \tilde A(s)\left(\tau t+J_0t_0+(J_0 u) (Q(t_0)+\tau^2 \tilde Q(t))\right)\\
        x(\lambda s)+w\tilde A(s)\left(Q(t_0)+\nabla Q(t_0)\cdot \tau t+\tau^2 \tilde Q(t)\right)
    \end{bmatrix}.
    \end{equation*}
     and so our argument is finished for the second coordinate. 

     \subsubsection{Simplifying vector coordinate, Part II}

    We now eliminate the constant term in $t$ in the vector coordinate. Using affine invariance with the first and the second coordinates, we can transform the vector coordinate to
    \begin{equation*}    
        \tau \tilde A(s)\left(t+\tau\tilde Q(t)(J_0 u+a_2v (J_0t_0)+a_2v (J_0u)Q(t_0) )\right),    
    \end{equation*}
    where we have used $\tilde A(s)=A(\lambda s)$ and \eqref{eqn:defn_tilde_v}.

    \fbox{Defining $\tilde u$}
    
    We now define
    \begin{equation}\label{eqn:defn_tilde_u}
        \tilde u=\tau(J_0 u+a_2v (J_0t_0)+a_2v (J_0u)Q(t_0) ),
    \end{equation}
    which we can assume to lie within $[-1,1]^l$ by another trivial partition if necessary. Rescaling by $\tau^{-1}$, we have transformed the vector coordinate to $\tilde A(s)(t+\tilde u\tilde Q(t))$.

    \subsubsection{Analysis of last coordinate}\label{sec:sigma_flat_rescaling}
    We have now reduced to
    \begin{equation}\label{eqn:reduction_2nd_coordinate}
    \begin{bmatrix}
        s_1\\
        s_2-\tilde v \tilde A(s)\tilde Q(t)\\        
        \tilde A(s)(t+\tilde u\tilde Q(t))\\
        x(\lambda s)+w\tilde A(s)\left(Q(t_0)+\nabla Q(t_0)\cdot \tau t+\tau^2 \tilde Q(t) \right)\\
    \end{bmatrix}.
    \end{equation}
    We may first apply an affine transformation to eliminate the affine terms in $t$ in the last coordinate, which then becomes
    \begin{equation*}
        x(\lambda s)+b \tilde A(s) \tilde Q(t),
    \end{equation*}
    where
    \begin{equation*}
        \begin{aligned}
            b&:=w\left(\tilde a_2\tilde v Q(t_0)-\tau\nabla Q(t_0)\cdot \tilde u+\tau^2\right)\\
            &=w\tau^2 \Big(a_2 vQ(t_0)-\nabla Q(t_0)\cdot (J_0 u+a_2v (J_0t_0)+a_2v (J_0u)Q(t_0) )+1\Big),
        \end{aligned}
    \end{equation*}
    using $\tilde A(s)=A(\lambda s)$, \eqref{eqn:defn_tilde_v} and \eqref{eqn:defn_tilde_u}. Since $c$ is small enough, we have $b\sim w\tau^2$; by a trivial partition again, we may just take $b=w\tau^2$. Thus, \eqref{eqn:reduction_2nd_coordinate} becomes
    \begin{equation}\label{eqn:reduction_2nd_coordinate_b=wsigma}
        \begin{bmatrix}
        s_1\\        
        s_2-\tilde v \tilde A(s)\tilde Q(t)\\
        \tilde A(s)(t+\tilde u\tilde Q(t))\\        
        x(\lambda s)+w\tau^2 \tilde A(s) \tilde Q(t)
    \end{bmatrix},
    \end{equation}

    Now it is time to apply the assumption of $\sigma$-flatness \eqref{eqn:sigma_flatness_rescaling_invariance}. In the reduced form \eqref{eqn:reduction_2nd_coordinate_b=wsigma}, this means that for all $(s,t)\in [-1,1]^{2+l}$ we have
    \begin{equation}\label{eqn:flatness_alpha}
    \begin{aligned}
        &|x(\lambda s)+w\tau^2 \tilde A(s) \tilde Q(t)\\
        &-\alpha_0-\alpha_1 s_1-\alpha_2 (s_2-\tilde v \tilde A(s)\tilde Q(t))-\alpha \cdot \tilde A(s)(t+\tilde u\tilde Q(t))|\le \sigma.
    \end{aligned}
    \end{equation}

    \fbox{Defining $\tilde x$}
    
    Taking $t=0$ in \eqref{eqn:flatness_alpha} gives
    \begin{equation}\label{eqn:flatness_t=0}
        |x(\lambda s)-\alpha_0-\alpha_1 s_1-\alpha_2 s_2|\le \sigma,\quad \forall s\in [-1,1]^2.
    \end{equation}
    We remark that since $\overline x:=x\circ \lambda$ is a polynomial, we can assume without loss of generality that $\alpha_0=\overline x(0)$ and $(\alpha_1,\alpha_2)=\nabla \overline x(0)$ (see \cite[Proposition 5.9]{LiYang2024}). In particular, we have
    \begin{equation}\label{eqn:Oct_08}
        \alpha_2=\eta \partial_2 x(c).
    \end{equation}    
    We then define
    \begin{equation}\label{eqn:defn_tilde_x}
        \tilde x(s):=\sigma^{-1}(\overline x(s)-\alpha_0-\alpha_1 s_1-\alpha_2 s_2),
    \end{equation}
    such that $\tilde x\in \mathcal P_{2,d}$. Plug \eqref{eqn:defn_tilde_x} back to the last coordinate of \eqref{eqn:flatness_alpha} to obtain
    \begin{equation*}
        \alpha_0+\alpha_1 s_1+\alpha_2 s_2+\sigma \tilde x(s)+w\tau^2 \tilde A(s)\tilde Q(t),
    \end{equation*}
    which, by affine invariance, further reduces to
    \begin{equation}
        \sigma \tilde x(s)+(\alpha_2 \tilde v+w\tau^2)\tilde A(s)\tilde Q(t).
    \end{equation}

    \fbox{Defining $\tilde w$}
    
    By the triangle inequality applied to \eqref{eqn:flatness_alpha} and \eqref{eqn:flatness_t=0}, we have
    \begin{equation*}
        |(w\tau^2+\alpha_2 \tilde v-\alpha\cdot \tilde u)\tilde Q(t)-\alpha\cdot t|\lesssim \sigma,\quad \forall t\in [-1,1]^l.
    \end{equation*}
    Since $|\det D^2 \tilde Q|\sim 1$, by Lemma \ref{lem:Hessian_bounded_below} below followed by triangle inequalities, we thus have $|w\tau^2+\alpha_2 \tilde v|\lesssim \sigma$.
    
    Thus, we may define
    \begin{equation}\label{eqn:defn_tilde_w}
        \tilde w:=\sigma^{-1}(\alpha_2 \tilde v+w\tau^2),
    \end{equation}
    such that $|\tilde w|\le 1$ by a trivial partition if needed. 

Then it can be easily observed from the affine transformations above that there exists an affine bijection $\Lambda:\R^{3+l}\to \R^{3+l}$ such that \eqref{eqn:Oct_08_01} and \eqref{eqn:Oct_08_02} hold. The exact expression of $\Lambda$ will not be used, so we leave it out.

\subsubsection{Analysis of special cases}
\begin{itemize}
    \item 
If $v=1$ and $\eta=\tau^2$, then using \eqref{eqn:defn_tilde_v}, we see $\tilde v=1$. Furthermore, when $v=1$, \eqref{eqn:reduction_2nd_coordinate_b=wsigma} becomes
\begin{equation*}
        \begin{bmatrix}
        s_1\\                
        s_2- \tilde A(s)\tilde Q(t)\\
        \tilde A(s)(t+\tilde u\tilde Q(t))\\
        \sigma \tilde x(s)+\sigma\tilde w \tilde A(s) \tilde Q(t)\\
    \end{bmatrix}.
    \end{equation*}
By affine invariance between the second and last coordinates and using \eqref{eqn:defn_tilde_w}, the above is equivalent to 
\begin{equation*}
        \begin{bmatrix}
        s_1\\                
        s_2- \tilde A(s)\tilde Q(t)\\
        \tilde A(s)(t+\tilde u\tilde Q(t))\\
        \sigma \tilde x(s)+ \sigma \tilde w s_2\\
    \end{bmatrix}.
    \end{equation*}
We may redefine
\begin{equation*}
    \hat x(s)=\tilde x(s) +\tilde w s_2, 
\end{equation*}
so that we have equivalently reduced to the case $\tilde w=0$. We also record that
\begin{equation}\label{eqn:compatibility_degen_det}
\begin{aligned}
    \partial_2 (\sigma\hat x(s))
    &=\partial_2 (\sigma\tilde x(s))+\sigma\tilde w\\
    &=\partial_2 \overline x(s)-\alpha_2+\sigma \tilde w\\
    &=\eta\partial_2 x(\lambda s),
\end{aligned}    
\end{equation}
where we have used \eqref{eqn:defn_tilde_x}, \eqref{eqn:defn_tilde_w}, \eqref{eqn:affine_lambda_S} and the assumption that $w=0$.

\item
If $\tau^2=\sigma$, $v=w=1$, and $\partial_2 x(c)=0$, then by \eqref{eqn:Oct_08} we have $\alpha_2=0$, and so by \eqref{eqn:defn_tilde_w} we have $\tilde w=1$. Also, by \eqref{eqn:defn_tilde_x} and \eqref{eqn:Oct_08}, we have
\begin{equation}
    |\partial_2 \tilde x(s)|=\sigma^{-1}\eta |\partial_2 x(\lambda s)-\partial_2 x(c)|.
\end{equation}
But by \eqref{eqn:defn_tilde_v}, we have $\tilde v=\sigma\eta^{-1}$. Thus, \eqref{eqn:Oct_14_10} follows.
 
\end{itemize}

This finishes the proof of Lemma \ref{lem:rescaling_invariance}.

\begin{lem}\label{lem:Hessian_bounded_below}
    Let $Q$ be a $C^2$ function on $[-1,1]^l$ such that some second derivative $|\partial_{ij}D^2Q(0)|\ge 1$. Then $\sup_{t\in [-1,1]^l}|Q(t)|\ge c$, where $c>0$ depends on $l$, the $C^2$ norm of $Q$, and the modulus of continuity of $D^2 Q$.
\end{lem}

\begin{proof}
The case $l=1$ follows easily from the fundamental theorem of calculus. For $l\ge 1$, the case $l=2$ is typical enough, and the higher dimensional cases follow from induction on $l$; we omit the details. 

For $l=2$, by symmetry, we only need to consider the cases $|Q_{11}(0)|\sim 1$ and $|Q_{12}(0)|\sim 1$.

If $|Q_{11}(0)|\sim 1$, then $|Q_{11}(t)|\sim 1$ for $t\in [-c,c]^l$ for some $c\sim 1$ by continuity. The the result follows from applying the one-dimensional case to $x_1\mapsto Q(x_1,0)$.

If $|Q_{12}(0)|\sim 1$, then $|Q_{12}(t)|\sim 1$ for $t\in [-c,c]^l$ for some $c\sim 1$ by continuity. By the fundamental theorem of calculus, we can write
\begin{align*}
    Q(c,c)-Q(0,c)-Q(c,0)+Q(0,0)=\int_0^c\int_0^c Q_{12}(t_1,t_2)dt_1dt_2,
\end{align*}
which has absolute value $\sim 1$. Thus, the result follows.
\end{proof}

\subsection{Reduction to $u=0$}\label{sec:u=0}

Lastly, we note that in some cases, we may apply a nonlinear change of variables $t':=t+uQ(t)$ to reduce to the case $u=0$. With this, we have $Q(t)=\tilde Q(t')$ for some $\tilde Q$ that lies in $\mathcal Q(\tilde c_1,\tilde c_2)$ (see \eqref{eqn:Q_small}) for possibly different $0<\tilde c_1\ll \tilde c_2\ll 1$. Unfortunately, the property $u=0$ is not invariant under rescalings, as can be seen from \eqref{eqn:defn_tilde_u}.

{\it Remark.} If we further assume that $Q$ has slightly higher regularity, say, $Q\in C^3$, then by the Morse Lemma (see, for instance, \cite{MilnorMorseLemma}), we may even take $Q(t)=c\sum_{i=1}^l e_i t_i^2$ where $e_i\in \{1,-1\}$. Even if we do not use Morse Lemma, within the context of decoupling, we can use Pramanik-Seeger (see \cite[Section 7]{BD2015}) to achieve the same reduction.

\section{Degeneracy determinant}\label{sec:degeneracy_determinant}
In this section, we introduce a degeneracy determinant $H$ defined on the pair $(\mathcal M,\mathfrak R)$, as in Definition \ref{defn:degeneracy_determinant} (recall \eqref{eqn:defn_family_M} and \eqref{eqn:mathfrak_R}). We begin with the following proposition.
\begin{prop}\label{prop:Gaussian_curvature_general}
    The absolute value of the Gaussian curvature of the parametrisation $\Phi$ in \eqref{eqn:defn_M} is comparable to
    \begin{equation}
        |\partial_2 x(s)|^l |\det D^2 x(s)|.
    \end{equation}
\end{prop}
\begin{proof}
    It suffices to assume $u=0$ as mentioned in Section \ref{sec:u=0} since a reparametrisation does not change the Gaussian curvature. By direct computation, a normal vector to the surface is given by
    \begin{equation*}
    \begin{bmatrix}
        -(1-a_2 LQ)\partial_1 x-a_1 \partial_2 x LQ\\
        -\partial_2 x\\        
        -\partial_2 x\nabla Q\\
        1-a_2 LQ
    \end{bmatrix},
    \end{equation*}
    where we recall $LQ(t)=Q(t)-t\cdot \nabla Q(t)$. The matrix of the second fundamental forms is given by
    \begin{equation*}
        \begin{bmatrix}
            (1-va_2 LQ)D^2 x & O\\
            O& (w+v\partial_2 x)A(s) D^2 Q
        \end{bmatrix}.
    \end{equation*}
    Since $1-va_2 LQ\sim 1$, $|\det D^2 Q|\sim 1 $ and $A(s)\sim 1$, the result follows.
\end{proof}

\begin{prop}
    $H\Phi:=\partial_2 x$ is a degeneracy determinant with respect to the pair $(\mathcal M,\mathfrak R)$.
\end{prop}
\begin{proof}
For Part \eqref{item:totally_degen}, suppose $\Phi\in \mathcal M$ is such that $\norm{H \Phi}_{L^\infty ([-1,1]^{2+l})}\le \sigma$, which means that $|\partial_2 x(s)|\le \sigma$ for all $s\in [-1,1]^2$. 

Define $y(s_1):=x(s_1,0)$, so $y\in \mathcal P_{1,d}$. Then we define $\Psi\in \mathcal M$ by
    \begin{equation}
        \Psi=(s_1,s_2- A(s)Q(t),  A(s)(t+uQ(t)), y(s_1)).
    \end{equation}
    We have $\mathcal R_{\tilde \Phi}=\mathcal R_{\Psi}$ since they have the same coordinate space parametrisation. Also, $H\Psi\equiv 0$, and by the mean value theorem, we see that $\Psi$ is within $O(\sigma)$ of $\Phi$. 
    
    Also, note that \eqref{eqn:regularity_degen_det} follows from \eqref{eqn:compatibility_degen_det}. This finishes the proof.

\end{proof}

\section{A general sublevel set decoupling}\label{sec:sublevel_set}
In this section, we state and prove the required sublevel set decoupling Assumption \ref{item:sublevel_set_decoupling} of Theorem \ref{thm:degeneracy_locating_principle}.

\begin{thm}\label{thm:general_sublevel_set}
Let $\Phi=\Phi_{A,Q,u,x}\in \mathcal M$ as in \eqref{eqn:defn_M}. Then for any $0<\delta\ll 1$, there exist a boundedly overlapping cover $\mathcal S_\delta$ of the sublevel set $\{s:|P(s)|<\delta\}$ by parallelograms $S\in \mathcal S$ of width at least $\delta$ (recall \eqref{eqn:affine_lambda_S}), and a partition $\mathcal T_\delta$ of $[-1,1]^l$ into squares $T$, such that the followings hold:
\begin{enumerate}    
    \item $l(T)^2=\eta_S\ge \delta$ (see \eqref{eqn:affine_lambda_S}).
    \item $|P|\lesssim \delta$ on each $S$.
    \item The set
    \begin{equation}\label{eqn:sublevel_set_S,T}
        G(S,T):=\{(s_1,s_2-A(s)Q(t), A(s)(t+uQ(t))):s\in S, t\in T\}
    \end{equation}
belongs to $\mathcal R_\Phi$ as in \eqref{eqn:defn_family_R}.

    \item For any $2\le p\le \frac{2(3+l)}{1+l}$, the set
    \begin{equation}\label{eqn:G_0}
        G_0:=\{(s_1,s_2-A(s)Q(t), A(s)(t+uQ(t)) ):(s,t)\in [-1,1]^{2+l},|P(s)|<\delta\}
    \end{equation}
    can be $\ell^p(L^p)$ decoupled into $G(S,T)$ at the cost of $O_\eps(\delta^{-\eps})$, where the implicit constant depends only on $d,l,\eps$.
\end{enumerate}
\end{thm}
Once Theorem \ref{thm:general_sublevel_set} is established, we apply this to the polynomial $P:=H\Phi=\partial_2 x$ to fulfill the requirement of \ref{item:sublevel_set_decoupling} of Theorem \ref{thm:degeneracy_locating_principle}. The required $\Xi$ is just $\lambda_{S\times T}$, and the surface $\Psi$.

The rest of this section is devoted to the proof of Theorem \ref{thm:general_sublevel_set}.

\fbox{Notation} In this section, we omit the dependence of constants on $d,l$. We also omit the decoupling exponents $\ell^p(L^p)$.

\subsection{Preliminary reduction}\label{sec:t'=t}

Before we start the main proof, we note that it suffices to prove the case $u=0$. Indeed, given $G_\delta(S,T):=G(S,T)\cap \{|P|<\delta\}$ and $G_0$ with a general $u$, denote $t'=\gamma(t)=t+uQ(t)$, $T':=\gamma(T)$, then
\begin{equation*}
    G_\delta(S,T)=\{(s_1,s_2-A(s)Q(\gamma^{-1} (t')), A(s)t'):s\in S, t'\in T',|P(s)|<\delta\}.
\end{equation*}
We may apply the decoupling for the case $u=0$, with $Q$ replaced by $Q\circ \gamma^{-1}$, to decouple $G_0$ into sets of the form
\begin{equation}\label{eqn:Mar_28_01}
    \{(s_1,s_2-A(s)Q(\gamma^{-1} (t')), A(s)t'):s\in S, t'\in T'',|P(s)|<\delta\},
\end{equation}
where $T''$ is a cube of side length $\eta_S^{1/2}$. But since $\gamma$ is bi-Lipschitz, this means that $\gamma^{-1}(T'')$ is also essentially a cube of side length $\eta_S^{1/2}$, which can then be slightly enlarged to become an actual square $T$ of side length $\sim \eta_S^{1/2}$. This is harmless since we can tolerate bounded overlap. Thus, \eqref{eqn:Mar_28_01} is enlarged to
\begin{equation*}
    \{(s_1,s_2-A(s)Q(t), A(s)(t+uQ(t))):s\in S, t\in T,|P(s)|<\delta\},
\end{equation*}
as required.

\subsection{Induction on degree}
We now adapt the argument of \cite[Section 4.3]{LiYang2023}, namely, induction on the degree $d$ of the polynomial $P$.

The base case when $d=0$ is trivial. Now let $d\ge 1$ and assume Theorem \ref{thm:general_sublevel_set} holds for all polynomials of degree at most $d-1$.

Let $P$ be a polynomial of degree at most $d$. Then the polynomials $\partial_1 P$, $\partial_2 P$ have degree at most $d-1$.

\subsection{Dyadic decomposition}

Perform a dyadic decomposition to $G_0$ to get the following sets:
\begin{equation}\label{eqn:G1G2G3}
    \begin{aligned}
        G_1&:=G_0\cap \{s\in [-1,1]^2:|\partial_1 P(s)|<\delta\},\\
    G_2&:=G_0\cap \{s\in [-1,1]^2:|\partial_1 P(s)|\sim \sigma_1,|\partial_2 P(s)|<\delta\},\\
    G_3&:=G_0\cap\{s\in [-1,1]^2:|\partial_1 P(s)|\sim \sigma_1,|\partial_2 P(s)|\sim \sigma_2\}.
    \end{aligned}
\end{equation}

\subsection{Applying induction hypothesis}
Now we apply the induction hypothesis to each set $G_i$, $i=1,2,3$ above to decouple them into rectangles.

Let us be more precise. For $G_1$, we apply the induction hypothesis to $\partial_1 P$ to get rectangles $S,T$ in the coordinate space with $\eta_S=l(T)^2$, such that each set of the form 
\begin{align*}
    \{(s_1,s_2-A(s)Q(t), A(s)t):s\in S,t\in T,|\partial_1 P(s)|<\delta\}
\end{align*}
is an almost rectangle, and such that on each $S,T$ we still have $|\partial_1 P(s)|\lesssim \delta$.

For $G_2$, we apply the induction hypothesis twice, to $\partial_1 P-\sigma_1$ and $\partial_2 P$ respectively, to get rectangles $S,T$ in the coordinate space with $\eta_S=l(T)^2$, such that each set of the form 
\begin{align*}
    \{(s_1,s_2-A(s)Q(t), A(s)t):s\in S,t\in T,|\partial_1 P(s)|\sim \sigma_1,|\partial_2 P(s)|<\delta\}
\end{align*}
is an almost rectangle, and such that on each $S,T$ we still have $|\partial_1 P|\sim \sigma_1$ and $|\partial_2 P|\lesssim \sigma_2$. (To ensure the former, we should actually be slightly more careful here about the ratio we choose in the dyadic decomposition; See \cite[Section 4.3.1]{LiYang2023} for details.)

The argument for $G_3$ is similar to that for $G_2$.

\subsection{Univariate cases}\label{sec:first_univariate_case}

We first decouple sets of the form $G_1$ and $G_2$ as in \eqref{eqn:G1G2G3}. For $G_1$, since we are at scale $\delta$ and $|\partial_1 P(s)|<\delta$, we may assume $P$ is a univariate polynomial in $s_2$. By the fundamental theorem of algebra, the set $\{s_2:|P(s_1,s_2)|<\delta\}$ is just a union of $O(1)$ intervals, so we may just locate to one such interval (without loss of generality, we may assume its length is at least $\delta$), and denote the resulting parallelogram by $S'$. Thus, by affine transformations, we are decoupling
\begin{equation*}
    \{(s_1, s_2-A(s)Q(t), A(s)t):s\in S',t\in T\}.
\end{equation*}
where we may also assume by rescaling that $S'$ is over $[-1,1]$. 
Since $\eta_{S'}\le \eta_S=l(T)^2$, by affine invariance, we just need to decouple at scale $\eta_{S'}$, for
\begin{equation*}
    \{(s_1,-\tilde A(s_1)Q(t),\tilde A(s_1)t):s\in [-1,1],t\in T\}
\end{equation*}
for some affine transformation $\tilde A$. But the decoupling of this follows exactly from Theorems \ref{thm:radial_principle} and \ref{thm:Bourgain_Demeter} which results in $T$ being decoupled to smaller cubes $T'$ satisfying $l(T')^2=\eta_{S'}$ (Indeed, this is essentially just decoupling for the light cone). But then the sets
\begin{equation*}
    \{(s_1, s_2-A(s)Q(t), A(s)t):s\in S',t\in T'\}
\end{equation*}
are essentially parallelograms, so we are done.

Now we decouple sets of the form $G_2$ as in \eqref{eqn:G1G2G3}. Since we are at scale $\delta$ and $|\partial_2 P(s)|<\delta$, we may assume $P$ is a univariate polynomial in $s_1$, and the proof of this case is similar to the first univariate case above. We omit the details.

\subsection{Main case}\label{sec:main_case}
Now we decouple sets of the form $G_3$ as in \eqref{eqn:G1G2G3}. This will be the main part of the proof of sublevel set decoupling. Compared with \cite[Section 4.3.3]{LiYang2023}, the main difficulty here is that we do not have symmetry between $s_1,s_2$. Thus, the cases $\sigma_1\le \sigma_2$ and $\sigma_1\ge \sigma_2$ are different in nature.

\subsubsection{Case $\sigma_1\lesssim \sigma_2$}\label{sec:sigma1<sigma2}
By rescaling in the $s_1$ direction, we can assume that $S$ is over $[-1,1]$. We may still assume $\sigma_1\ge \delta$ after the rescaling, since otherwise it reduces to Section \ref{sec:first_univariate_case}.

\begin{lem}\label{lem:level_approx}
There are $O(1)$ many subintervals $I_i\sub [-1,1]$ and algebraic functions $g_i$ over $I_i$ such that for some suitable degree constants $C_1,C_2$ we have
\begin{align*}
    &\{(s_1,s_2):|P(s_1,s_2)| <\delta \} \cap \Omega_0\\
    &\sub \bigsqcup_i \{(s_1,s_2):s_1\in I_i, |s_2-g_i(s_1)|\leq C_1\delta \sigma_2^{-1}\}\cap \Omega_0\\
    &\sub \{(s_1,s_2):|P(s_1,s_2)| \le C_2 \delta \}\cap \Omega_0.
\end{align*} 
\end{lem}
\begin{proof}
    This follows from the implicit function theorem similar to Proposition \ref{prop:exist_g}. We omit the details since it is almost the same as the proof of \cite[Proposition 4.5]{LiYang2023}.
\end{proof}
Now we perform a little reduction. First, by Lemma \ref{lem:level_approx}, we now assume $i=1$ and there is only one $I=I_i$ and $g=g_i$. By rescaling we may assume $I=[-1,1]$, and we also pretend the set $\{|P|<\delta\}\cap S=\{(s_1,g(s_1)+u):s_1\in [-1,1],|u|\lesssim\sigma_2^{-1}\delta\}$ to simplify notation. So $g$ is a pseudo-polynomial we studied in Section \ref{sec:pseudo-polynomials}. Note also that $\eta_S=\sigma_2^{-1}\delta$.

By Proposition \ref{prop:I_2}, $g'$ and all its derivatives are essentially bounded above by $\sigma_1/\sigma_2$. We now need to decouple the set
\begin{align*}
    \{(s_1,s_2-A(s)Q(t), A(s)t):s_1\in [-1,1],|s_2-g(s_1)|\lesssim \eta_S,t\in T\}.
\end{align*}
But this is essentially the $\eta_S$-neighbourhood of the surface
\begin{equation}\label{eqn:Sep_13_01}
    \{(s_1,g(s_1)-(1+a_1s_1+a_2g(s_1))Q(t), (1+a_1s_1+a_2g(s_1)) t):s_1\in [-1,1], t\in T\}.
\end{equation}
To this end, we perform a dyadic decomposition according to $|g''|<\eta_S$ or $|g''|\sim \sigma\ge \eta_S$. By Proposition \ref{prop:fta_pseudo}, we may locate to an interval $I$. 

\fbox{$|g''|<\eta_S$}

In this case, we see $g$ is affine at scale $\eta_S$, so by affine invariance, \eqref{eqn:Sep_13_01} essentially becomes
\begin{equation*}
    (s_1,-\Lambda(s_1)Q(t),  \Lambda(s_1) t),
\end{equation*}
for some bounded affine transformation $\Lambda$ with $\Lambda(0)=1$, which can be dealt with using Theorems \ref{thm:radial_principle} and \ref{thm:Bourgain_Demeter}. Note that $T$ is decoupled into cubes of side length $\eta_S^{1/2}$.

\fbox{$|g''|\sim \sigma$}

If $\sigma\sim 1$, then we can compute that the Gaussian curvature of the surface in \eqref{eqn:Sep_13_01} has its absolute value essentially $1$. Thus, Theorem \ref{thm:Bourgain_Demeter} applies, and $T$ is decoupled into cubes of side length $\eta_S^{1/2}$.

If $\sigma\ll 1$, we argue as follows. Assume by rescaling that $I=[-1,1]$, which is allowed by Proposition \ref{prop:rescaling_pseudo}. Then we define
\begin{equation*}
    f(s):=\sigma^{-1}(g(s)-g(0)-g'(0)s),
\end{equation*}
so that by Corollary \ref{cor:I_2}, $f$ is a pseudo-polynomial with $|f''|\sim 1$. By affine invariance, \eqref{eqn:Sep_13_01} is equivalent to 
\begin{equation}\label{eqn:Sep_14_01}
    (s,\sigma f(s)-(\Lambda(s)+a_2\sigma f(s))Q(t), (\Lambda(s)+a_2 \sigma f(s))t),
\end{equation}
for some bounded affine transformation $\Lambda$ with $\Lambda(0)=1$. We then apply an intermediate decoupling at scale $\sigma$, for the surface
\begin{equation*}
    (s,-\Lambda (s)Q(t),\Lambda (s)t),
\end{equation*}
which, by Theorems \ref{thm:radial_principle} and \ref{thm:Bourgain_Demeter}, gives squares of side length $\sigma^{1/2}$. Fix such a square with centre $t_0$. Then define 
\begin{equation*}
    \tilde Q(t):=\sigma^{-1}(Q(t)-Q(t_0)-\nabla Q(t_0)\cdot (t-t_0)),
\end{equation*}
which lies in $\mathcal Q$, and so \eqref{eqn:Sep_14_01} is rescaled to
\begin{equation*}
    (s,f(s)-(\Lambda(s)+a_2\sigma f(s))\tilde Q(t),(\Lambda(s)+a_2 \sigma f(s))(\sigma^{1/2}t+t_0) ).
\end{equation*}
A further affine transformation turns the above into
\begin{equation*}
    (s, f(s)-(\Lambda(s)+a_2\sigma f(s))\tilde Q(t),(\Lambda(s)+a_2 \sigma f(s))(\sigma^{1/2}t+a_2 \sigma t_0 \tilde Q(t)),
\end{equation*}
which can be rescaled to
\begin{equation*}
    (s,f(s)-(\Lambda(s)+a_2\sigma f(s))\tilde Q(t),(\Lambda(s)+a_2 \sigma f(s))(t+a_2 \sigma^{1/2} t_0 \tilde Q(t) ).
\end{equation*}
Since $\sigma\ll 1$ and $|f''|\sim 1$, a direct computation shows that this surface has Gaussian curvature $\sim 1$. Thus, we apply Theorem \ref{thm:Bourgain_Demeter} to finish the proof, and $T$ is decoupled into cubes of side length $\eta_S^{1/2}$.

\subsubsection{Case $\sigma_2 \ll \sigma_1$}
This case is more difficult.

\underline{Preliminary reductions}

We first rescale in the $s_1$ direction so that $S$ is over $[-1,1]$, and we may still assume $\sigma_2\ll \sigma_1$, otherwise it reduces to Section \ref{sec:sigma1<sigma2}. We then apply a shear transform and a translation in the $s_2$ direction, so that we may assume $S$ is of the form 
\begin{equation*}
    [-1,1]\times [-\eta,\eta],
\end{equation*}
where $\eta=\eta_S$. The relation $\sigma_2\ll \sigma_1$ still holds after this shear transform.

We then apply Lemma \ref{lem:level_approx} with the roles of $s_1$ and $s_2$ switched, so that we may assume the set $\{|P|<\delta\}\cap S=\{(g(s_2)+u,s_2):|u|<\kappa,\,\,s_2\in I_0\}$, where $g$ is a pseudo-polynomial we studied in Section \ref{sec:pseudo-polynomials}, $\kappa:=\delta \sigma_1^{-1}$, and $I_0\sub [-\eta,\eta]$. By Proposition \ref{prop:I_2}, $g'$ and all its derivatives are essentially bounded. Now we need to decouple the following set:
\begin{equation}\label{eqn:Oct_12}
    \{(s_1,s_2-A(s)Q(t), A(s)t):|s_1-g(s_2)|\lesssim \kappa,s_2\in I_0, t\in T\}.
\end{equation}
Note that $|I_0|\le \eta=l(T)^2$. 

\underline{Decoupling at scale $|I_0|$}

At scale $|I_0|$, by affine invariance, \eqref{eqn:Oct_12} essentially becomes (where $c_0$ is the centre of $I_0$)
\begin{equation*}
    \{(s_1,-\Lambda(s_1)Q(t),  \Lambda(s_1) t):|s_1-g(c_0)|\lesssim \max\{\kappa, |I_0|\}, t\in T\},
\end{equation*}
for some bounded affine transformation $\Lambda$ with $\Lambda(0)=1$. Since $|I_0|\le l(T)^2$, we apply Theorems \ref{thm:radial_principle} and \ref{thm:Bourgain_Demeter} to decouple $T$ into cubes $T'$ of side length $|I_0|^{1/2}$. Thus, \eqref{eqn:Oct_12} becomes
\begin{equation}\label{eqn:Oct_11_01}
    \{(s_1,s_2-A(s)Q(t), A(s)t):|s_1-g(s_2)|\lesssim \kappa,s_2\in I_0, t\in T'\}.
\end{equation}
If $|I_0|\le \kappa$, then the set \eqref{eqn:Oct_11_01} is already equivalent to a rectangle, so we are done. Thus, for the rest of the proof, we assume $|I_0|> \kappa$. Then our problem becomes decoupling the $\kappa$-neighbourhood of the surface
\begin{equation*}
    \{(g(s_2),s_2-(1+a_1g(s_2)+a_2 s_2) Q(t),(1+a_1g(s_2)+a_2s_2) t):s_2\in I_0, t\in T'\}.
\end{equation*}

\underline{Dyadic decomposition}

We dyadically decompose $I_0$ according to $|g'(s_2)|\le \kappa$ or $|g'(s_2)|\sim \sigma> \kappa$, and apply Proposition \ref{prop:fta_pseudo} to reduce to an interval $I\sub I_0$. 

We apply the same argument as decoupling $T$ above, now at scale $|I|$, to decouple $T'$ into cubes $T''$ of side length $|I|^{1/2}$. Thus, we now need to 
decouple at scale $\kappa$ for the surface:
\begin{equation}\label{eqn:Oct_11_02}
    \{(g(s_2),s_2-(1+a_1g(s_2)+a_2 s_2) Q(t),(1+a_1g(s_2)+a_2s_2) t):s_2\in I, t\in T''\}.
\end{equation}
If $|I|\le\kappa$ or $|g'(s_2)|\le \kappa$ for $s_2\in I$, then the set \eqref{eqn:Oct_11_02} is already equivalent to a rectangle, so we are done. Thus, we are left with the proof when $|g'(s_2)|\sim \sigma> \kappa$ and $|I|> \kappa$.

\underline{Rescaling parameter space}

We now rescale $I$ to $[-1,1]$ and $T''$ to $[-1,1]^l$, in a similar way to the proof of Lemma \ref{lem:rescaling_invariance} (followed by the treatment of Section \ref{sec:u=0}). Since $l(T'')^2=|I|$, decoupling \eqref{eqn:Oct_11_02} at scale $\kappa$ is essentially decoupling
\begin{equation}\label{eqn:Oct_11_04}
    \{(\tilde g(s_2),s_2- (\tilde \Lambda(s_2)+\tilde g(s_2))\tilde Q(t),(\tilde \Lambda(s_2)+\tilde g(s_2))t ): s_2\in [-1,1],t\in [-1,1]^l\},
\end{equation}
for some bounded affine $\tilde \Lambda$ with $\tilde \Lambda(0)=1$, $\tilde Q\in \mathcal Q$, and some pseudo-polynomial $\tilde g$ (by Proposition \ref{prop:rescaling_pseudo}) satisfying $|\tilde g'|\sim \sigma':=\sigma |I|$.

We need to decouple the surface \eqref{eqn:Oct_11_04} at scale $\kappa$, such that the parameter space is decoupled into rectangles $\tilde S\times \tilde T$ satisfying $\eta_{\tilde S}=l(\tilde T)^2$.

If $\sigma'\le \kappa$, then \eqref{eqn:Oct_11_04} is equivalent to a rectangle, so there is no further decoupling to be done. Thus, we may assume $\sigma'\ge \kappa$.

\underline{Reducing to Case $\sigma_1\lesssim \sigma_2$ }

We define
\begin{equation*}
    f(s_2):=(\sigma')^{-1}(\tilde g(s_2)-\tilde g(0)),
\end{equation*}
which is a pseudo-polynomial that satisfies $|f'|\sim 1$, by Proposition \ref{prop:g'_sim_sigma}. Thus, decoupling \eqref{eqn:Oct_11_04} at scale $\kappa$ essentially becomes decoupling the following surface at scale $\kappa (\sigma')^{-1}$:
\begin{equation}
    \{(f(s_2),s_2- (\overline \Lambda(s_2)+\sigma' f(s_2))\tilde Q(t),(\overline \Lambda(s_2)+\sigma' f(s_2))t ): s_2\in [-1,1],t\in [-1,1]^l\},
\end{equation}
for some bounded affine $\overline \Lambda$ with $\overline \Lambda(0)=1$.

We thus change $s=f(s_2)$, so $s_2=h( s_1)$ for some pseudo-polynomial $h$, by Proposition \ref{prop:inverse_pseudo}. Thus, we need to decouple
\begin{equation*}
    \{(s, h( s)-(\overline \Lambda (h(s))+\sigma' s)\tilde Q(t),(\overline \Lambda (h(s))+\sigma' s)t):s\in [-1,1],t\in [-1,1]^l\}.
\end{equation*}
But this is essentially the same as \eqref{eqn:Sep_13_01} studied in Section \ref{sec:sigma1<sigma2}, which results in the parameter space being decoupled into rectangles $\tilde S\times \tilde T$ satisfying $\eta_{\tilde S}=l(\tilde T)^2$. So we are done.

\section{Totally degenerate decoupling}\label{sec:totally_degenerate_case}
In this section, we study decoupling for the surface $\Phi=\Phi_{A,Q,u,x}$ in \eqref{eqn:defn_M} when $H\Phi=\partial_2 x\equiv 0$, which corresponds to \ref{item:degnerate_decoupling} of Theorem \ref{thm:degeneracy_locating_principle}.

Since $\partial_2 x\equiv 0$, we may write $x(s)=y(s_1)$, where $y\in \mathcal P_{1,d}$. Thus, we are decoupling at scale $\delta$ for
\begin{equation*}
    \{(s_1,s_2-A(s)Q(t),A(s)(t+uQ(t)),y(s_1)):(s,t)\in [-1,1]^{2+l}\}.
\end{equation*}

We now project to the first and last coordinates and apply uniform decoupling for polynomial curves (see \cite[Theorem 1.4]{Yang2}) to get $s_1\in I$ on which $y$ is $\delta$-flat. Since $\eta_{I\times [-1,1]}=1$, we have $\bar\Phi(I\times [-1,1]\times [-1,1]^{l})\in \mathcal R_\Phi$. This finishes the proof.

For the lower dimensional decoupling, we consider two cases, namely, $\eta_S=w$ or $m_{11}=w$, where $w$ is the width of $S$.

If $\eta_S=w$, then $l(T)^2=w$, so decoupling at scale $w$ reduces to decoupling $(s_1,x(s_1+s_{1,0},m_{21}s_1+s_{2,0}))$ (refer to \eqref{eqn:affine_lambda_S}) at scale $w$, which then follows from \cite[Theorem 1.4]{Yang2}. This gives parallelograms in $\mathcal R_\Phi$ with width at least $w^{1/2}$.

If $m_{11}=w$, then decoupling at scale $w$ essentially reduces to 
\begin{equation*}
    (s_1,s_2-A(s)Q(t),A(s)(t+uQ(t)),x(s_2)).
\end{equation*}
We use decoupling for the cone to deal with the coordinate plane, and then apply \cite[Theorem 1.4]{Yang2} to deal with decoupling for $(s_2,x(s_2))$. This gives parallelograms in $\mathcal R_\Phi$ with width at least $w^{1/2}$. We omit the details.

\section{Nondegenerate decoupling}\label{sec:nondegenerate_case}

In this section, we study decoupling for the surface $\Phi=\Phi_{A,Q,u,x}$ 
\begin{equation*}
    \{(s_1,s_2-A(s)Q(t),A(s)(t+uQ(t)),x(s)):(s,t)\in [-1,1]^{2+l}\},
\end{equation*}
where $|H\Phi|=|\partial_2 x(s)|\ge K^{-1}$ in $[-1,1]^{2+l}$, which corresponds to \ref{item:nondegnerate_decoupling} of Theorem \ref{thm:degeneracy_locating_principle}. The main idea is a Pramanik-Seeger iteration.

\fbox{Notation} In this section, we omit the dependence of constants on $d,l$.

\subsection{Reducing to $|\partial_2 x|\sim 1$}
We first reduce to the case $|\partial_2 x|\sim 1$, with a loss not more than $K^{O(1)}$.

\subsubsection{Partition and dyadic decomposition}
We first apply a dyadic decomposition, obtaining subsets $|\partial_2 x|\sim \kappa$ where $\kappa\ge K^{-1}$.

\subsubsection{Sublevel set decoupling}
We apply Theorem \ref{thm:general_sublevel_set} at scale $\kappa$ to decouple the coordinate space into parallelograms of the form $\bar\Phi_{A,Q,u}(S\times T)$ with $\eta_{S}=l(T)^2$.

\subsubsection{Rescaling}
We rescale $S$ and $T$ to $[-1,1]^2$ and $[-1,1]^l$, respectively, using the same argument in the proof of Lemma \ref{lem:rescaling_invariance} (before Section \ref{sec:sigma_flat_rescaling}). The rescaled surface now becomes of the form
 \begin{equation}\label{eqn:Oct_14}
    \{(s_1,s_2-\tilde A(s)\tilde Q(t),\tilde A(s)(t+\tilde u\tilde Q(t)),  \tilde x( s) ):(s,t)\in [-1,1]^{2+l}\}.
\end{equation}
Now $|\partial_2 \tilde x|\sim \kappa':=\kappa \eta_S$ over $[-1,1]^2$, so we can write
\begin{equation*}
    \tilde x(s)=y(s_1)+\kappa' z(s_1,s_2),
\end{equation*}
for some polynomials $y,z$ with bounded coefficients.

\subsubsection{Decoupling at scale $\kappa'$}
We decouple at scale $\kappa'$, so that the surface \eqref{eqn:Oct_14} becomes indistinguishable with
 \begin{equation*}
    \{(s_1,s_2-\tilde A(s)\tilde Q(t),\tilde A(s)(t+\tilde u\tilde Q(t)),  y(s_1) ):(s,t)\in [-1,1]^{2+l})\}.
\end{equation*}
By projecting into the first and the last coordinates, it suffices to decouple the polynomial curve $(s_1,y(s_1))$. By \cite[Theorem 1.4]{Yang2}, we obtain intervals $I\sub [-1,1]$.

\subsubsection{Rescaling}
Let $\overline y(s_1):=(\kappa')^{-1}(\lambda_I s_1-\text{ affine terms})\in \mathcal P_{1,d}$. We rescale $I$ to $[-1,1]$ in \eqref{eqn:Oct_14} and multiply the last coordinate by $(\kappa')^{-1}$ to obtain
\begin{equation*}
    \{(s_1,s_2-\bar A(s)\tilde Q(t),\bar A(s)(t+\tilde u\tilde Q(t)),  \overline x(s) ):(s,t)\in [-1,1]^{2+l})\},
\end{equation*}
where $\overline x(s):=(\kappa')^{-1} \overline y(s_1)+z(\lambda_I s_1,s_2)\in \mathcal P_{2,d}$. Note that now $|\partial_2 \overline x|\sim 1$. We have thus completed the reduction, at the total cost of $C_\eps K^\eps$. 

Abusing notation, it thus remains to study decoupling for the surface
\begin{equation}\label{eqn:Oct_14_01}
    \{(s_1,s_2-A(s)Q(t),A(s)(t+uQ(t)), x(s)):(s,t)\in [-1,1]^{2+l})\},
\end{equation}
where $|\partial_2 x(s)|\sim 1$ on $[-1,1]^2$.

\subsection{Affine transformation}
We still need another treatment before we use the Pramanik-Seeger iteration. By affine equivalence between the second and the last coordinates, we may transform \eqref{eqn:Oct_14_01} into
\begin{equation*}
    \{(s_1,s_2-A(s)Q(t),A(s)(t+uQ(t)), \overline x(s)+wA(s)Q(t)):(s,t)\in [-1,1]^{2+l}\},
\end{equation*}
where $\overline x(s):=x(s)-\partial_2 x(0)s_2$, and $w:=\partial_2 x(0)$. Note that $\partial_2 \overline x(0)=0$ and $|w|\sim 1$. We may of course also assume $\overline x(0)=0$ and $\partial_1 x(0)=0$ for free. Abusing notation and writing $w=1$, we may thus study decoupling for the surface
\begin{equation}\label{eqn:PS_ready}
    \{(s_1,s_2-A(s)Q(t),A(s)(t+uQ(t)), x(s)+A(s)Q(t)):(s,t)\in [-1,1]^{2+l}\},
\end{equation}
where $x$ has no affine terms, and such that
\begin{equation}\label{eqn:Oct_14_11}
    |1+\partial_2 x(s)|\sim 1,\quad \forall s\in [-1,1]^2.
\end{equation}

\subsection{Pramanik-Seeger iteration}

We state the decoupling theorem we would like to prove.
\begin{thm}\label{thm:nondegenerate_decoupling_PS}
    Let $A\in \mathcal A$, $Q\in \mathcal Q$, $u\in [-1,1]^l$, $x\in \mathcal P_{2,d}$ with no affine terms and obeying \eqref{eqn:Oct_14_11}, and let $\Phi=\Phi_{A,Q,u,x}$ be as in \eqref{eqn:PS_ready}. Then for $\eps>0$ 
 and $0<\delta\ll 1$, there exist a boundedly overlapping cover $\mathcal S_\delta$ of $[-1,1]^2$ by $(x,\delta)$-flat parallelograms $S$ with width at least $\delta$, and a partition $\mathcal T_\delta$ of $[-1,1]^l$ by cubes $T$ of side length $\delta^{1/2}$, such that for every $2\le p\le \frac{2(l+4)}{l+2}$, the parallelogram $[-1,1]^{2+l}$ can be $\Phi$-$\ell^p(L^p)$ decoupled into the boundedly overlapping parallelograms $\bar\Phi(S\times T)$, at the cost of $C_\eps\delta^{-\eps}$ for every $\eps>0$.
\end{thm}
The rest of this section is devoted to the proof of Theorem \ref{thm:nondegenerate_decoupling_PS}.

\subsubsection{Trivial partition}
Given $\eps,\delta\ll 1$. Losing $\delta^{-\eps}$, we may trivially decouple $[-1,1]^2$ into smaller squares $S_0$ of side length $\delta^\eps$, and decouple $[-1,1]^l$ into smaller cubes $T_0$ of side length $\delta^{\eps/2}$.

Denote by $D(\delta)$ the best constant of the required decoupling in Theorem \ref{thm:nondegenerate_decoupling_PS}, except that $[-1,1]^2$ and $[-1,1]^l$ are replaced by $S_0$ and $T_0$, respectively. It suffices to show that $D(\delta)\lesssim \delta^{-\eps}$.

\subsubsection{Applying decoupling at larger scale}
Let $\sigma:=\delta^{1-\eps}$. We will apply the definition of $D(\sigma)$ to obtain $S'\in \mathcal S_\sigma$ and $ T'\in \mathcal T_\sigma$ as prescribed in the theorem. We need to further decouple at scale $\delta$, now for the surface
\begin{equation}\label{eqn:Oct_14_02}
    \{(s_1,s_2-A(s)Q(t),A(s)(t+uQ(t)), x(s)+A(s)Q(t)):s\in S',t\in T'\}.
\end{equation}

\subsubsection{Affine invariance}
In view of \eqref{eqn:Oct_14_11}, by affine invariance between the second the last coordinates in \eqref{eqn:Oct_14_02}, we may assume that $\partial_2 x(s_0)=0$, where $s_0$ is the centre of $S'$.

\subsubsection{Rescaling}

We use exactly the same rescaling as in the proof of Lemma \ref{lem:rescaling_invariance} to rescale $ S'$ to $[-1,1]^2$ and $ T'$ to $[-1,1]^l$, arriving at decoupling 
\begin{equation}\label{eqn:Oct_11_PS}
    \{(s_1,s_2-\tilde v \tilde A(s)\tilde Q(t),\tilde A(s)(t+\tilde u\tilde Q(t)),\tilde x(s)+\tilde w \tilde A(s)\tilde Q(t)  ):(s,t)\in [-1,1]^{2+l}\}
\end{equation}
at scale $\delta \sigma^{-1}=\delta^\eps$. Note that we are in the case $v=1$, $w=1$, $\tau^2=l(T)^2=\sigma$. Thus, the last statement of Lemma \ref{lem:rescaling_invariance} says that $\tilde w=1$, and we have
\begin{equation*}
    \sup_{s\in [-1,1]^2}|\tilde v \partial_2 \tilde x(s)|=\sup_{s\in S'} |\partial_2 x(s)-\partial_2 x(c)|.
\end{equation*}
But since $S'\sub S_0$, it has diameter at most $\delta^\eps$. Thus, 
\begin{equation}\label{eqn:Oct_14_12}
    \sup_{s\in [-1,1]^2}|\tilde v \partial_2 \tilde x(s)|\lesssim \delta^\eps.
\end{equation}

\subsubsection{Approximation}

We now let $t'=\tilde A(s)(t+\tilde u\tilde Q(t))$, and $s_2':=s_2-\tilde v \tilde A(s)\tilde Q(t)$, so that \eqref{eqn:Oct_11_PS} becomes
\begin{equation*}
    (s_1,s_2',t',\tilde x(s_1,s_2'+\tilde v \tilde A(s)\tilde Q(t))+ w\tilde A(s) \tilde Q(t)   ).
\end{equation*}
By direct computation using \eqref{eqn:Oct_14_12}, we see that the last coordinate is within $O(\delta^\eps)$ of $\tilde x(s_1,s_2')+Q(t')$.

\subsubsection{Applying known decoupling result}

We can now apply, for instance, Proposition \ref{prop:refine_IJ_smooth} at scale $\delta^{\eps}$ to the following surface
\begin{equation*}
    (s_1,s_2',t',\tilde x(s_1,s_2')+\tilde Q(t')),
\end{equation*}
so that the coordinate space $[-1,1]^{2+l}$ is decoupled into parallelograms of the form $S\times T$, where each $S$ is $(\tilde x,\delta^\eps)$-flat, and $T$ is a cube of side length $(\delta^{\eps})^{1/2}$, at the cost of $C_\eps\delta^{-\eps^2}$. Moreover, we can guarantee that $\eta_S\ge l(T)^2=\delta^\eps$, by \cite[Theorem 5.22]{LiYang2024}. Changing variables back, this ensures that $\bar \Phi'(S\times T)\in \mathcal R_{ \Phi'}$, where 
\begin{equation*}
    \bar\Phi'(s_1,s_2,t):=(s_1,s_2+\tilde v\tilde A(s)\tilde Q(t),\tilde A(s)(t+\tilde u\tilde Q(t))).
\end{equation*}
Note that we treated $t'$ as if it were $t$, which does not do any harm, since the change of variables between $t'$ and $t$ is a diffeomorphism such that the $C^2$ norms of both itself and its inverse are bounded by $O(1)$, and the $t'$ space is decoupled into cubes. A more rigorous treatment is similar to Section \ref{sec:t'=t}.

\subsubsection{Iteration}

The above argument implies a bootstrap inequality of the form:
\begin{equation*}
    D(\delta)\le D(\delta^{1-\eps}) C_{\eps} \delta^{-\eps^2}. 
\end{equation*}
Iterating this bootstrap inequality for $\sim \eps^{-1}$ times gives us $D(\delta)\lesssim_\eps \delta^{-2\eps}$. This finishes the proof of Theorem \ref{thm:nondegenerate_decoupling_PS} and thus nondegenerate decoupling.

\section{Decoupling for homogeneous functions of nonzero degree}\label{sec_homo_nonzero}
In this section, we prove Theorem \ref{thm:homo}. We first prove the general case, and discuss the special convex case at the end of this section.

Let $\phi(x_1,x_2,x_3)$ be homogeneous of degree $\alpha\ne 0$ in the annulus $\{x\in \R^3:1\le |x|\le 2\}$, namely, $\phi(rx)=r^\alpha \phi(x)$.

\subsection{Preliminary reductions}\label{sec:reduction_homo_nonzero}
First, we may partition the annulus into several pieces, and assume $x_3\sim 1$ by relabelling. Next, write $P(x_1,x_2)=\phi(x_1,x_2,1)$, and thus we may write
\begin{equation}\label{eqn_Apr_30_01}
    \phi(x_1,x_2,x_3)=x_3^\alpha \phi\left(\frac {x_1} {x_3},\frac {x_2}{x_3},1\right)=x_3^\alpha P\left(\frac {x_1} {x_3},\frac {x_2} {x_3}\right).
\end{equation}
We then change variables
\begin{equation*}
    r=x_3,\quad s_1=\frac{x_1} {x_3},\quad s_2=\frac{x_2} {x_3},
\end{equation*}
and so we can rewrite
\begin{equation*}
    (x_1,x_2,x_3,\phi(x_1,x_2,x_3))=(rs_1,rs_2,r,r^\alpha P(s_1,s_2)).
\end{equation*}
Lastly, we can apply the polynomial approximation argument in Section 2 of \cite{LiYang2023} and assume $P\in \mathcal P_{2,d}$. Then it suffices to prove a uniform decoupling result, with the cost of decoupling independent of the coefficients of $P$.

\fbox{Notation} All implicit constants below in this section are allowed to depend on $d,\alpha$.

\subsection{A sublevel set decoupling}
We will need an easier analogue of the sublevel set decoupling in Theorem \ref{thm:general_sublevel_set}, which decouples sublevel sets of the form
\begin{equation*}
    H_\delta:=\{(rs_1,rs_2,r):r\in [1,2], (s_1,s_2)\in [-1,1]^2,|P(s_1,s_2)|<\delta\}.
\end{equation*}

\begin{thm}\label{thm:homo_sublevel_set}
Let $P\in \mathcal P_{2,d}$. For $0<\delta\ll 1$, there exists a boundedly overlapping cover $\mathcal S_\delta$ of $\{s\in [-1,1]^2:|P(s)|<\delta\}$ by rectangles $S$ with width at least $\delta$, such that the following holds:
\begin{enumerate}    
    \item $|P|\lesssim \delta$ on each $S$.

    \item For every $2\le p\le 6$ and every $\eps\in (0,1]$, the set $H_\delta$ can be $\ell^2(L^p)$ (and hence $\ell^p(L^p)$) decoupled into
    \begin{equation*}
        H_\delta(S):=\{(rs_1,rs_2,r):r\in [1,2], (s_1,s_2)\in S\}
    \end{equation*}
    at the cost of $C_\eps \delta^{-\eps}$, where $C_\eps$ depends only on $\eps,d,\alpha$.
    \end{enumerate}
\end{thm}
To prove this theorem, we will follow a similar proof to that of \cite[Section 4.3]{LiYang2023}, and we just give the main ideas. 

\subsubsection{Case of implicit functions}
We first solve the case when $|\partial_1 P|\sim 1$ over $[-1,1]^2$. In this case, the set $|P(s_1,s_2)|<\delta$ is essentially the $\delta$-neighbourhood of the graph of a pseudo-polynomial $s_2=g(s_1)$, and so we just need to decouple the $\delta$-neighbourhood of
\begin{equation*}
    \{(rs_1,rg(s_1),r):r\in [1,2]:s_1\in [-1,1]\}.
\end{equation*}
Now we relabel the variables $t=s_1$, $s=r\in [1,2]$, $r(s)=s$ and rearrange the variables, so that the surface above becomes
\begin{equation*}
    (s,st,sg(t)).
\end{equation*}
But then this can be handled by Theorem \ref{thm:radial_principle} and \cite[Theorem 1.5]{Yang2}.

\subsubsection{Induction on degree}
We induce on the degree $d$ of the polynomial $P$. The case $d=0$ is trivial. Assume the decoupling result holds for polynomials $P$ of degree at most $d-1$.

Let $P$ have degree at most $d$. Then the polynomials $\partial_1 P$ and $\partial_2 P$ have degree at most $d-1$, so we may apply the induction hypothesis to decouple into rectangles $S$ satisfying one of the following three conditions:
\begin{align*}
    &\{|\partial_2 P|<\delta\}\\
    &\{|\partial_1 P|<\delta,\quad |\partial_2 P|\sim \sigma_2\}\\
    &\{|\partial_1 P|\sim \sigma_1,\quad |\partial_2 P|\sim \sigma_2\}
\end{align*}
where $\sigma_1,\sigma_2\ge \delta$ are dyadic constants.

The case $|\partial_2 P|<\delta$ is easy to handle, since it again reduces to the surface $(rs_1,rg(s_1),r)$ which we just studied. The case $|\partial_1 P|<\delta$ is the same by symmetry. 

For the main case $|\partial_1 P|\sim \sigma_1,|\partial_2 P|\sim \sigma_2$, we may assume by symmetry that $\sigma_1\le \sigma_2$. Then we perform a rotation to assume $S$ is axis parallel, and in $S$ we have either $|\partial_1 P|\sim \sigma_2$ or $|\partial_2 P|\sim \sigma_2$. Assume again the latter by symmetry. Then we can rescale in the $s_2$ direction to reduce to the case $|\partial_2 P|\sim 1$, using a similar proof to that of \cite[Section 4.3.3]{LiYang2023}. This finishes the proof of Theorem \ref{thm:homo_sublevel_set}.

\subsection{Proof of decoupling for homogeneous functions}\label{sec_homo}
Recall we have reduced to decoupling the surface
\begin{equation*}
    \{(rs_1,rs_2,r,r^\alpha P(s_1,s_2)):r\in [1,2],(s_1,s_2)\in [-1,1]^2\},
\end{equation*}
where $P\in \mathcal P_{2,d}$.

\fbox{Notation} Unless otherwise specified, all decoupling exponents below in this section are $\ell^p(L^p)$, where $2\le p\le \frac {10}3$.

\subsubsection{Dyadic decomposition}
We first dyadically decompose according to whether $|P|<\delta$ or $|P|\sim \sigma$ for some dyadic constant $\sigma\ge \delta$.

For the case $|P|<\delta$, we first apply Theorem \ref{thm:homo_sublevel_set} to decouple $H_\delta$ into rectangles $H_\delta(S)$, such that on each $S$ we have $|P|\lesssim \delta$, so that the set
\begin{equation*}
    \{(rs_1,rs_2,r,r^\alpha P(s_1,s_2)):r\in [1,2],(s_1,s_2)\in S\}
\end{equation*}
is already $O(\delta)$-flat.

For the case $|P|\sim \sigma$, our goal is to reduce to the case $\sigma= 1$. Indeed, we first apply Theorem \ref{thm:homo_sublevel_set} to $P-\sigma$ (assume by symmetry that $P$ is positive) to obtain rectangles $S$ on each of which $|P|\sim \sigma$. We then rescale $S$ to $[-1,1]^2$, and so the polynomial $\sigma^{-1}P\circ \lambda_S$ has bounded coefficients, and satisfies $|\sigma^{-1}P\circ \lambda_S|\sim 1$.

\subsubsection{Change of variables}
Now we just need to decouple the surface
\begin{equation*}
    \{(rs_1,rs_2,r,r^\alpha P(s_1,s_2)):r\in [1,2],(s_1,s_2)\in [-1,1]^2\}.
\end{equation*}
when $|P|\sim 1$ in $[-1,1]^2$. By symmetry, we just study when $P\sim 1$.

Below we will perform several changes of variables. First, we let 
\begin{equation*}
    u=r^\alpha P(s_1,s_2), \quad \beta=\alpha^{-1},\quad Q=P^{-\beta},
\end{equation*}
and move the last coordinate to the first, so that the surface becomes
\begin{equation}\label{eqn:Apr_2}
    \{(u,u^{\beta} s_1 Q(s),u^{\beta} s_2Q(s), u^{\beta} Q(s)):u\sim 1, s\in [-1,1]^2\},
\end{equation}
where $Q\sim 1$. Our next goal is to write the above in the form of $(u,u^\beta t,u^\beta\psi(t))$ for suitable $\psi(t)=\psi(t_1,t_2)$.

First, by a trivial decoupling, we partition $[-1,1]^2$ into $O(1)$ smaller squares of length $\ll 1$. By affine invariance, we may locate to the smaller square centred at $0$, so that $|s|\ll 1$. With this, we define
\begin{equation*}
    t_1=s_1 Q(s_1,s_2),\quad t_2=s_2 Q(s_1,s_2),
\end{equation*}
and we hope to rewrite $s$ as a function of $t$. To this end, consider the function $\Phi:\R^4\to \R^2$ defined by
\begin{equation*}
    \Phi(t_1,t_2,s_1,s_2)=(s_1 Q(s_1,s_2)-t_1,s_2 Q(s_1,s_2)-t_2).
\end{equation*}
The Jacobian matrix $\left(\frac{\partial \Phi}{\partial s}\right)$ is equal to
\begin{equation*}
    \begin{bmatrix}
        Q+s_1\partial_1 Q & s_1 \partial_2 Q\\
        s_2 \partial_1 Q & Q+s_2 \partial_2 Q
    \end{bmatrix}.
\end{equation*}
Since $Q\sim 1$ and $|s|\ll 1$, its determinant $\sim 1$. This means that we can also write $s_1,s_2$ to be implicit functions of $t_1,t_2$. With this, denote
\begin{equation}\label{eqn:defn_psi_homo}
    \psi(t_1,t_2):=Q(s_1(t_1,t_2),s_2(t_1,t_2)),
\end{equation}
so that \eqref{eqn:Apr_2} becomes
\begin{equation}\label{eqn:before_twice_approximation}
    \{(u,u^\beta t_1,u^\beta t_2,u^\beta \psi(t_1,t_2)):u\sim 1,|t|\le c\},
\end{equation}
where $c\ll 1$. Note that $L\psi(t)=\psi(t)-t\cdot \nabla \psi(t)\sim 1$, using $Q\sim 1$ and $|t|\ll 1$.

For some $c=c_{d,\beta}>0$, define a family $\mathcal Q_{d,\beta}$ of real-analytic functions $\psi$ as follows:
\begin{equation}
\begin{aligned}
    \mathcal Q_{d,\beta}:=&\left\{\psi|\psi:[-c,c]^2\to \R \text{ is defined via }(sP(s)^{-\beta},P(s)^{-\beta})=(t,\psi(t)),\right.\\
    &\left.P\in \mathcal P_{2,d},\,\,  P(s)\ge c_{d,\beta}\right\}.
\end{aligned}
    \end{equation}
It suffices to prove decoupling for surfaces of the form \eqref{eqn:before_twice_approximation}, for functions $\psi\in \mathcal Q_{d,\beta}$ uniform in the coefficients of $P\in \mathcal P_{2,d}$.

Given $\eps,d$. By another polynomial approximation, we may assume that $\psi\in \mathcal P_{2,d'}$, for $d'$ being the smallest integer greater than $\eps^{-1}+1$.

\subsubsection{Applying radial principle}\label{sec_homo_apply_conical}

Applying Theorem \ref{thm:radial_principle} with $r(s)=s^\beta$, we can reduce to decoupling
\begin{equation*}
    (u,t_1,t_2,u^\beta+\psi(t_1,t_2)).
\end{equation*}
This entails us to check Condition \ref{condition:decoupling}, which follows from Corollary \ref{cor:decoupling_condition_bivariate}, but with $\phi_1(x)$ and $\phi_2(y)$ replaced by $\psi(t_1,t_2)$ and $u^\beta$, respectively. Note that the function $u\mapsto u^\beta$ can be embedded into a family of functions with nonzero second derivative on $[1,2]$.

\subsection{Convex/concave case}
In this last subsection, we outline the reason why we can obtain $\ell^2(L^p)$ decoupling, $2\le p\le \frac {10}3$, when $\phi$ defines a convex surface.

When $\phi$ is concave, we have $\beta\in (0,1]$, and the eigenvalues of $D^2\psi$ must be nonpositive. Thus, $s^\beta$ is concave, and so $s^\beta+\psi(t)$ is also concave. Thus, we obtain $\ell^2$ decoupling. The case where $\phi$ is convex is similar.

\section{Appendix}\label{sec:appendix}
In this appendix, we discuss further generalisations of the uniform decoupling theorem for bivariate polynomials proved in \cite{LiYang2023}, in addition to the generalisation made in Theorem \cite[Theorem 5.22]{LiYang2024}).

\begin{prop}\label{prop:refine_IJ_smooth}
      Let $n\ge 3$ be an integer. Let $2\le p\le \frac{2(n+1)}{n-1}$, $0 < \varepsilon < 1$, $0<\delta < \sigma \ll  1$. Let $\phi\in \mathcal P_{2,d}$ and let $Q\in C^2([-1,1]^{n-3})$ satisfy $\det D^2 Q\ne 0$.
      
      Let $P_\sigma\sub [-1,1]^2$ be a $(\phi,\sigma)$-flat parallelogram, and let $Q_\sigma\sub [-1,1]^{n-3}$ be a cube of side length $\sigma^{1/2}$. Let $\mathcal Q_\delta(Q_\sigma)$ denote a tiling of $Q_\sigma$ by cubes of side length $\delta^{1/2}$. 

Then there exists a family of $(P,\delta)$-flat parallelograms $\mathcal P_\delta(P_\sigma) = \mathcal P_{\delta}(P_\sigma,\phi,\varepsilon)$ such that the following holds: 
\begin{enumerate}
    \item $P_\sigma\times Q_\sigma$ can be $(\phi(s)+Q(t))$-$\ell^p(L^p)$ decoupled into $\mathcal{K}_\delta(P_\sigma,Q_\sigma)$ at cost $ O_\varepsilon( \sigma^\eps\delta^{-\eps})$, where
    \begin{equation*}
        \mathcal K_\delta(P_\sigma,Q_\sigma):=\{P_\delta\times Q_\delta:P_\delta\in \mathcal P_\delta(P_\sigma),\,\,Q_\delta\in \mathcal Q_\delta(Q_\sigma)\}.
    \end{equation*}
    Moreover, if $D^2 \phi$ are $D^2 Q$ are positive-semidefinite, then this $\ell^p(L^p)$ decoupling can be upgraded to $\ell^2(L^p)$ decoupling. 
    \item $\cup \mathcal P_{\delta}$ covers $P_\sigma$ and is contained in $2P_\sigma$.
    \item Each $P_\delta\in \mathcal P_\delta(P_\sigma)$ has width at least $\delta$. Also, the overlap function of $\mathcal P_\delta$ is $O(1)$; in particular, $\#\mathcal P_\delta(P_\sigma)\lesssim\delta^{-2}|P_\sigma|$, so in particular, $\#\mathcal K_\delta(P_\sigma,Q_\sigma)\lesssim\delta^{-O(1)}$.
    \item For each $P_\delta \in \mathcal P_\delta(P_\sigma)$, we have
    \begin{equation}
         \sigma^{-1} (P_\sigma - c(P_\sigma)) \subseteq \delta^{-1}(P_\delta - c(P_\delta)),
    \end{equation}
    where $c(P_\sigma)$ and $c(P_\delta)$ are the centres of $P_\sigma$ and $P_\delta$ respectively. In particular, taking $P_\sigma=[-1,1]^2$, this shows that $P_\delta$ has width at least $\delta$.
\end{enumerate}
All implicit constants here (and the proof below) may only depend on the parameters $n,d,p,\eps$, $\norm{D^2 Q}_{\infty}$ and $\inf |\det D^2 Q|$.
\end{prop}

\begin{proof}

    Since $P_\sigma$ is $(\eta,\sigma)$-flat, it is also $(\phi,O(\sigma))$-flat. Thus, the coefficients of 
    \begin{equation*}
        \tilde \phi(x):=\sigma^{-1}(\phi(\lambda_{P_\sigma} x)-\phi(c)-\nabla \phi(c)\cdot \lambda_{P_\sigma} x)
    \end{equation*}
    are essentially bounded above by $1$. Similarly, we define
    \begin{equation*}
        \tilde Q(y):=\sigma^{-1}Q(\lambda_{Q_\sigma}y-Q(c')-\nabla Q(c')\cdot \lambda_{Q_\sigma}y),
    \end{equation*}
    where $c'$ is the centre of $Q_\sigma$. Note that $\|D^2\tilde Q\|_\infty=\norm{D^2 Q}_\infty$ and $\inf |\det D^2 \tilde Q|=\inf |\det D^2  Q|$. Rescaling by $\sigma$, the problem becomes the $(\tilde \phi(x)+\tilde Q(y))$-decoupling of $[-1,1]^{n-1}$ at scale $\sigma^{-1}\delta$ into parallelograms $\tilde P_\delta\times \tilde Q_\delta$. Combining \cite[Theorem 5.22]{LiYang2024} and a simpler argument than the proof of \cite[Corollary 1.5]{LiYang2024} (i.e. $J=1$), the cost of this decoupling is $O_{\eps}(\sigma^\eps\delta^{-\eps})$, and each parallelogram $\tilde P_\delta$ has width at least $\sigma^{-1}\delta$. Thus, if we define $P_\delta:=\lambda_{P_\sigma}(\tilde P_\delta)$, then we have
    \begin{equation*}
        \sigma^{-1}(P_\sigma-c(P_\sigma))\sub \delta^{-1} (P_\delta-c(P_\delta)).
    \end{equation*}
 \end{proof}   
Proposition \ref{prop:refine_IJ_smooth} has the following useful corollary, which allows us to verify Condition \ref{condition:decoupling}.
\begin{cor}\label{cor:decoupling_condition_bivariate}
    When $k=1,2$, $\phi_1\in \mathcal P_{k,d}$ and $\det D^2\phi_2\ne 0$, Condition \ref{condition:decoupling} holds with $\ell^p(L^p)$ decoupling, where $2\le p\le \frac{2(k+l+2)}{k+l}$. If $D^2\phi_1$ and $D^2\phi_2$ are positive-semidefinite, then this $\ell^p(L^p)$ decoupling can be upgraded to $\ell^2(L^p)$ decoupling.
\end{cor}

\bibliographystyle{alpha}
\bibliography{reference}

\end{document}